\documentclass{amsart}

\makeatletter
 \def\@textbottom{\vskip \z@ \@plus 8pt}
 \let\@texttop\relax
\makeatother
\usepackage{amsthm,amsmath,amssymb,enumerate,xcolor,mathtools,multirow,url,hyperref,graphicx,caption,subcaption}
\usepackage[utf8]{inputenc}
\newcommand{\QQ}{\mathbf{Q}}
\newcommand{\HH}{\mathbf{H}}
\newcommand{\CC}{\mathbf{C}}
\newcommand{\ZZ}{\mathbf{Z}}
\newcommand{\PP}{\mathbf{P}}

\newcommand{\FF}{\mathbf{F}}
\newcommand{\calL}{\mathcal{L}}
\renewcommand{\div}{\mathrm{div}}
\newcommand{\Gal}{\mathrm{Gal}}
\newcommand{\F}{H_\tau[\curve]}
\newcommand{\Fx}{H_\tau[x]}
\newcommand{\Fj}{H_\tau[j]}
\newcommand{\Ff}{H_\tau[f]}
\newcommand{\HC}{H_\tau[\curve]}
\newcommand{\Hx}{H_\tau[x]}

\newcommand{\Cl}{\mathrm{Cl}}
\newcommand{\D}{\mathcal{D}}
\newcommand{\mc}{\mathcal}
\newcommand{\fraka}{\mathfrak{a}}
\newcommand{\xnul}[1]{X^0(#1)}
\newcommand{\xnulN}{\xnul{N}}
\newcommand{\xplus}[1]{X^0_+(#1)}
\newcommand{\xplusN}{\xplus{N}}

\newcommand{\curve}{C}

\DeclareMathOperator{\ord}{ord}
\DeclareMathOperator{\SL}{SL}

\numberwithin{equation}{section}

\theoremstyle{definition}
\newtheorem{definition}[equation]{Definition}
\newtheorem{theorem}[equation]{Theorem}
\newtheorem{proposition}[equation]{Proposition}
\newtheorem{remark}[equation]{Remark}
\newtheorem{example}[equation]{Example}
\newtheorem{lemma}[equation]{Lemma}

\begin{document}
	
	\title{Generalized class polynomials}
	
	\author{Marc Houben}
	\address{Marc Houben. Departement Wiskunde, KU Leuven,
		Celestijnenlaan 200B -- bus 2400,
		3001 Leuven,
		Belgium\vspace{-0.25cm}}
	\address{imec-COSIC, KU Leuven, Kasteelpark Arenberg 10/2452, 3001 Leuven, Belgium \vspace{-0.25cm}}
	\address{Mathematisch Instituut, Universiteit Leiden,
		P.O.\ box 9512,
		2300 RA Leiden,
		The Netherlands}
	\email{marc.houben@kuleuven.be}
	
	\author{Marco Streng}
	\address{Marco Streng. Mathematisch Instituut, Universiteit Leiden,
		P.O.\ box 9512,
		2300 RA Leiden,
		The Netherlands}
	\email{streng@math.leidenuniv.nl}
	\urladdr{https://www.math.leidenuniv.nl/~streng}
	
	\begin{abstract}
	
	The Hilbert class polynomial has as roots the $j$-invariants of elliptic curves
	whose endomorphism ring is a given imaginary quadratic order.
	It can be used to compute elliptic curves over finite fields
	with a prescribed number of points.
	Since its coefficients are typically rather large,
	there has been continued interest in finding alternative modular functions
	whose corresponding class polynomials are smaller.
	Best known are Weber's functions,
	which reduce the size by a factor of 72
	for a positive density subset of imaginary quadratic discriminants.
	On the other hand, Bröker and Stevenhagen showed that
	no modular function will ever do better
	than a factor of~100.83.
	We introduce a generalization of class polynomials,
	with reduction factors that are not limited by the Bröker-Stevenhagen bound.
	We provide examples matching Weber's reduction factor.
	For an infinite family of discriminants,
	their reduction factors surpass those of
	all previously known
	modular functions
	by a factor at least 2.
	\end{abstract}
	
	\maketitle

	\section{Introduction}
	
	The \emph{Hilbert class polynomial} $H_D[j]$
	of the imaginary quadratic order $\mathcal{O}$ of discriminant $D$
	is the minimal polynomial of the $j$-invariant
	of an elliptic curve with endomorphism ring~$\mathcal{O}$.
	It is a defining polynomial of the ring class field of $\mathcal{O}$
	and can be used for constructing elliptic curves over a finite field with a given number of points.
	Its coefficients are however rather large, which limits its practical usefulness.
	Already in 1908, Weber~\cite{weber3} therefore introduced alternative
	\emph{class invariants} to be used instead of $j$,
	which resulted in \emph{class polynomials}
	with coefficients that have roughly $1/72$ of the digits of the coefficients
	of the Hilbert class polynomial for certain discriminants.
	
	There has been continued interest in alternative class invariants
	ever since (e.g.~\cite{birch1969,
		schertz1976,
		gee-stevenhagen1998,
		gee1999,
		schertz2002weber,
		enge-morain2002,
		enge-schertz2004,
		enge-schertz2005,
       broker-stevenhagen-constructing,
       enge-sutherland2010,
		enge-schertz2013,
		enge-morain2014}).
	None however matched, let alone surpassed, the factor $72$ of Weber's
	functions.
	Moreover, Br\"oker and Stevenhagen~\cite{broker-stevenhagen-constructing} showed
	that no class invariant will ever do better than a factor~100.83.
	Under Selberg's eigenvalue conjecture \cite[Conjecture~1]{Selberg}, this bound reduces to~96.
	
	We introduce \emph{generalized (multivariate) class polynomials},
	define an appropriate notion of their \emph{reduction factor},
	and show that this notion indeed gives a measure of their ``size''
	compared to the Hilbert class polynomial (Section \ref{sec:reductionfactors}).
	Contrary to classical class polynomials,
	the reduction factors of generalized class polynomials are
	not limited by the Br\"oker-Stevenhagen bound.
	
	We give a family of generalized class polynomials
	for which we prove that the reduction factor matches
	Weber's $72$ for a large range of
	values of~$D$, including infinitely many values of $D$ where
	no reduction of $36$ or better was previously known
	(Section~\ref{sec:x0}).
	We also give an example that possibly 
	achieves the factor~$120$ (Remark~\ref{rem:120}).
	
	Though the focus of this paper is on introducing the generalized class invariants
	and studying their height, we also give a preliminary analysis indicating
    that the height reduction leads to a speed-up in their computation
    (Section~\ref{sec:runtime}),
    and we show how to use them for constructing elliptic curves over finite fields (Section~\ref{sec:apply}).
    
	\subsection*{Acknowledgements}
	
	The authors would like to thank
	Karim Belabas, Peter Bruin, Andreas Enge, David Kohel, Filip Najman, and Andrew Sutherland
	for helpful discussions and Edgar Costa for helping us use the LMFDB.
	The authors would further like to thank the 
	organisers Samuele Anni,
	Valentijn Karemaker, and Elisa Lorenzo Garc\'ia of
	AGC${}^2$T 2021 for the inspiring setting in which the first ideas
	for this project came to be.
	Finally, the authors would like to thank the anonymous referees for their many comments
	that led to improvements of the exposition.
	The first-listed author is supported by the Research Foundation -- Flanders (FWO) under a PhD Fellowship Fundamental Research.

	\section{Generalized class polynomials}\label{sec:section2}
	
\begin{definition}\label{def:modularcurve}
By a \emph{modular curve over $\QQ$}
we mean a smooth, projective, geometrically irreducible
curve $\curve$ over $\QQ$ together with a map $\psi : \HH\rightarrow \curve(\CC)$
from the upper half space $\HH\subset \CC$
with the following property.
There exists a positive integer $N$
such that for every function $f\in \QQ(\curve)$, the function
$f\circ \psi$ is a modular function for $\Gamma(N)$
with all $q$-expansion coefficients
in $\QQ^{\mathrm{ab}}$.

We identify $f$ with $f\circ \psi$
and we identify $\psi$ with the induced morphism of curves
$X(N)\rightarrow C$.
\end{definition}

For an order $\mc{O}$ in an imaginary quadratic number field $K$, 
we denote by $K_{\mc{O}}$ the associated ring class field.
Let $f$ be a modular function and $\tau\in\HH$ imaginary quadratic,
say a root of $aX^2+bX+c$ for coprime integers $a,b,c$.
The pair $(f,\tau)$ is called a \emph{class invariant} for the imaginary
quadratic order $\mc{O}=\ZZ[a\tau]$ if $f(\tau)$ lies in the ring class field~$K_{\mc{O}}$.
The \emph{discriminant} $D$ of the class invariant is the
discriminant of $\mc{O}$.
The Galois group $G$ of $K(f(\tau))/K$ is isomorphic
via the Artin map
to a quotient of the Picard group $\Cl(\mc{O})$.
Associated to a class invariant is its minimal polynomial over~$K$,
also known as the \emph{class polynomial},
\begin{equation*}
\Ff := \prod_{\sigma\in G}\big(X-\sigma(f(\tau))\big)\quad\in K[X].
\end{equation*}
Under additional restrictions, class polynomials
can sometimes be shown to have coefficients in $\QQ$ (cf.\ \cite[Thm.\ 4.4]{enge-morain2014}, 
\cite[Thm.\ 5.4]{enge-streng}); in that case we call the class polynomials \emph{real}.
Oftentimes, a modular function admits class invariants for an infinite family of discriminants, determined by a certain congruence condition
(\cite{schertz2002weber}, \cite[Thm.\ 4.3]{enge-morain2014}).
Sometimes the discriminant uniquely determines the class polynomial for a given modular function.

\begin{example}
The modular $j$-function admits a unique class polynomial for any discriminant $D<0$, called the \emph{Hilbert class polynomial} $H_D[j]:=H_\tau[j]$.
It can be seen as a function on $\PP^1$ whose zeros are
the $j$-invariants of elliptic curves with CM by the imaginary
quadratic order of discriminant $D$ and whose poles are restricted to the point at infinity.
\end{example}

We propose a generalization of class polynomials, seen as functions on modular curves of higher genus, for which the classical class polynomials can be viewed as the genus zero case. We will mostly restrict ourselves to the case of genus one, as this will make notation considerably less complicated. We discuss the arbitrary genus case in Section~\ref{sec:general}.
Let $C$ be a modular curve over $\QQ$ with a smooth Weierstrass model
$y^2+a_1xy+a_3y=x^3+a_2x^2+a_4x+a_6$, and suppose that $(x,\tau),(y,\tau)$ are class invariants for some imaginary quadratic $\tau\in\HH$.
Consider $G = \Gal(K(x(\tau),y(\tau))/K)$ and $m=\#G$.
If we denote by $\mc{D}$ the divisor of the unique point at infinity of $C$, then $\mc{L}(\infty \mc{D})$ has a basis $b_0=1,b_1=x,b_2=y,b_3=x^2,b_4=xy,b_5=x^3,b_6=x^2y,\ldots$ (ordered by ascending degree).
There exist $a_i\in K$, not all zero, such that
\begin{equation}\label{eq:lineareqn}
\sum_{i=0}^m a_i b_i(\tau)=0.
\end{equation}
In fact, up to scaling by an element of $K^{\times}$, there exists a unique function $F_{\tau}[C]=\sum_{i=0}^m a_ib_i\in K(C)$ such that
\begin{equation}\label{eq:defF}
\div F_{\tau}[C] = \left[\sum_{\sigma\in G} \left(\sigma(\psi(\tau))\right)\right]+\left(-\sum_{\sigma\in G}\sigma(\psi(\tau))\right)-(m+1)\mc{D}.
\end{equation}
\begin{definition}\label{def:genclasspoly1}
We call $F_{\tau}[C]$ as in \eqref{eq:defF}
a \emph{generalized class function} for~$\tau$.
The associated \emph{generalized class polynomial}
is the unique $H_{\tau}[C]\in K[X,Y]$ of degree $\leq 1$ in $Y$
such that $H_{\tau}[C](x,y) = F_{\tau}[C]$.
\end{definition}
We note that the
polynomial
$H_{\tau}[C]$ depends
on the choice of $x$ and~$y$, but we leave this out of the notation.
In Section~\ref{sec:general} (and
in particular Definition~\ref{def:genclasspoly2})
we will allow more general
divisors~$\mc{D}$
and bases~$\mc{B}$,
leading to more general
functions $F_{\tau}[C,\mc{B}]$ and polynomials
$H_{\tau}[C,\mc{B}]$.

\begin{definition}
We call the point
$P = \sum_{\sigma \in G}\sigma(\psi(\tau))\in C(K)$
the \emph{Heegner point} of the class function $F$.
\end{definition}
If the Heegner point $P$ is the point at infinity, then $a_m=0$.
Otherwise, the point $-P$ is a zero of~$F$.
In particular, if $P=-(0,0)$, then $a_0=0$.

For $N\in\ZZ_{>0}$, we denote by $\xnulN$ the smooth,
projective, geometrically irreducible curve over $\QQ$ with function field consisting of the modular functions for the modular group
$\Gamma^0(N)=\{\begin{psmallmatrix}
a & b\\
c & d
\end{psmallmatrix} \in \SL_2(\ZZ)\mid b\equiv 0 \pmod{N}\}$
that have rational $q$-expansion.
We denote by $\xplusN$ the quotient of $\xnulN$ by the
Fricke-Atkin-Lehner involution
$z\mapsto -N/z$, 
and write $\eta(z)$ for the Dedekind $\eta$-function
\begin{equation*}
\eta(z) = q^{1/24}\prod_{n=1}^{\infty}(1-q^n),
\quad\mbox{where}\quad
q = \exp(2\pi i z).
\end{equation*}

\begin{example}\label{example:main}
Consider the genus one modular curve $C:=\xplus{119}$.
Its conductor as an elliptic curve is~$17$
(Cremona label \href{https://www.lmfdb.org/EllipticCurve/Q/17/a/4}{17a4})\footnote{One way to deduce this is as follows.
	Using the command \texttt{J0(119).decomposition()} in SageMath~\cite{sage}
	one finds that $C$ has conductor~$17$.
	For each of the Weierstrass models of the now finitely many possible curves~\cite{lmfdb},
	there are finitely many options for the divisor of the function $\mathfrak{w}_{7,17}$
	given by \eqref{eq:defw}. The curve $C$ has two rational CM points
	(both of discriminant $-19$), so given a possible Weierstrass model
	together with a possible divisor for $\mathfrak{w}_{7,17}$, one
	can first determine $\mathfrak{w}_{7,17}$ as a function of the
	Weierstrass coordinates $x,y$ by evaluating in one CM point, and then determine whether it has the expected value in the other CM point. This process excludes all but one of the options, and we at once in fact deduce both the Weierstrass model \eqref{eq:weierstrass} and the relation between $\mathfrak{w}_{7,17}$ and $x$ and $y$ \eqref{eq:relwxy}.}.
A Weierstrass model for $E$ is given by\footnote{We note that a slightly ``simpler'' Weierstrass model $v^2+uv+v=u^3-u^2-u$ exists by taking $u=x$ and $v=-y-2x$, but the given model \eqref{eq:weierstrass} turns out to yield slightly better practical reduction factors (see Section~\ref{ssec:x0119}).}
\begin{equation}\label{eq:weierstrass}
y^2+3xy-y=x^3-3x^2+x,
\end{equation}
where $x,y\in\QQ(C)$ have respective $q$-expansions
\begin{eqnarray*}
x &=& q^{-2} + q^{-1} + 1 + q + 2q^2 + 2q^3 + 3q^4 + 3q^5 + 4q^6 + 5q^7 + \ldots,\\
y &=& q^{-3} + 1 + 2q + 2q^2 + 4q^3 + 4q^4 + 7q^5 + 9q^6 + 12q^7 +\ldots,\\
& & \mbox{where this time $q = \exp(2\pi i z/119)$.}
\end{eqnarray*}
The ``double eta quotient'' $\mathfrak{w}_{7,17}$ given by
\begin{equation}\label{eq:defw}
\mathfrak{w}_{7,17}(z)
=\frac{\eta(z/7)\eta(z/17)}{\eta(z)\eta(z/119)}
\end{equation}
is invariant under the action of $\Gamma^0(N)$ \cite[Thm.\ 1]{newman1}
and the Fricke-Atkin-Lehner involution \cite[Thm.\ 2]{enge-schertz2005}, hence also forms an element of the (rational) function field of $C$. It is related to $x$ and $y$ by
\begin{equation}\label{eq:relwxy}
\mathfrak{w}_{7,17}=-y+x^2-x.
\end{equation}
The curve $\xplus{119}$ has two cusps, and they are both rational.
In the given Weierstrass model,
these correspond to the point $(0,0)$ and the point at infinity.
Numerical examples of generalized class polynomials specifically for $\xplus{119}$ are given in Section~\ref{ssec:x0119}.
We will treat this curve as our main test case in the rest of the paper.
\end{example}
		
\section{Estimates and reduction factors}\label{sec:reductionfactors}

\subsection{Reduction factors}

We define  the \emph{reduction factor}
of a modular curve $\curve$
to be
\begin{equation}\label{eq:reductiondef}
	r(\curve) = \frac{\deg(j : X(N)\rightarrow \PP^1)}{\deg(\psi : X(N) \rightarrow \curve)}.
\end{equation}
In the case $\curve =\PP^1$, we denote this number also by~$r(\psi)$
and our notation and terminology coincide with 
that of~\cite{broker-stevenhagen-constructing}.
The number $r(\psi)^{-1}$ is denoted by
$\widehat{c}(\psi)$
in \cite{enge-morain2002}
and by $c(\psi)$ in~\cite{enge-morain2014}.
Br\"oker and Stevenhagen \cite[Theorem~4.1]{broker-stevenhagen-constructing}\footnote{The
arXiv version v1 of \cite{broker-stevenhagen-constructing} has weaker bounds than the final
publication and needs to be combined with \cite[Appendix~2]{Kim} to get the same result.}
show $r(\psi)\leq 32768/325 \leq 100.83$.
Under Selberg's eigenvalue conjecture,
one can even prove $r(\psi)\leq 96$.
The best known $\psi$ achieves $r(\psi) = 72$.
This result does not however apply directly to $r(\curve)$.
For example, we have
\begin{equation}\label{eq:reductionx0}
	r(\xnulN) = N\prod_{p\mid N} (1+\frac{1}{p})
\quad\mbox{and}\quad
r(\xplusN) = \frac{1}{2} r(\xnulN)
\quad\mbox{if}\quad N>1.
\end{equation}
Our main example $\curve = \xplus{119}$ therefore achieves
$r(C) = \frac{1}{2}(7+1)(17+1)=72$. For (hyper)elliptic modular curves $\curve$
we get $r(C) \leq 201.65$
(or $r(C) \leq 192$ 
under Selberg's eigenvalue conjecture), by applying the bounds to the $x$-function.
Surprisingly, all elliptic curve quotients of $\xnulN$
we found so far have $r\leq 72$ (Section~\ref{sec:searchelliptic}).
In Section~\ref{sec:general} we will discuss
higher-genus curves, which allow for unbounded
$r(\curve)$.

\begin{remark}\label{rmk:besidesrc}
In the applications we have in mind, 
the reduction factor is the main source of improvement in computational efficiency.
It is important to note, however, that this number $r(\curve)$ does not tell the complete story,
even in the ``classical'' setting ($\curve\cong\PP^1$), for example for the following reasons.
\begin{enumerate}
\item There are many challenges when computing class polynomials, and even more with
generalized class polynomials. See Section~\ref{sec:runtime}.
\item In the CM method (Section~\ref{sec:apply}), we will want to find a $j$-invariant in $\FF_p$ from a point in $C(\FF_p)$.
This is done using the minimal polynomial of the $j$-function over $\QQ(C)$,
known as the~\emph{modular polynomial} (Lemma~\ref{lem:modularpolynomials}).
This works best if the degree of $j$ over $\QQ(C)$ is small.
For example, this degree is $1$ for $C = \xnulN$, is $2$ for $C = \xplusN$,
and ranges from $1$ to $20$ in \cite[Table~7.1]{enge-morain2014},
making $\xplus{119}$ a good choice in this respect.
\item If the (generalized) class polynomial is not real, then its coefficients lie in an imaginary quadratic extension of $\QQ$; roughly doubling its bit size.
This issue can be avoided by imposing additional restrictions on $\curve$ or $\tau$,
see Sections \ref{ssec:x0nramified} and \ref{ssec:x0nplus}.
\end{enumerate}
On the other hand, there are two important tricks that may be used in complementary directions,
providing computational improvements beyond the reduction factor $r(\curve)$:
\begin{enumerate}
\item Under some constraints, typically when all primes dividing the level of the modular curve ramify in the CM field, both the degree and height of the class polynomial are cut in half. This happens for example in the
record-computation of \cite{enge-sutherland2010} for the Atkin invariant $A_{71}$ when $71$ divides the discriminant,
leading to class polynomials that are $2^2\cdot 36 = 144$ times smaller than the Hilbert class polynomial (note that the reduction factor $r(A_{71})$ is $36$ in this case).
The same trick also applies to generalized class polynomials, see Section \ref{ssec:x0nplusramified},
which in the case of $\xplus{119}$ leads to a factor $2^2\cdot 72 = 288$ in size reduction.
\item When the class number is composite, one can decompose the ring class field into a tower of fields whose defining polynomials have smaller degrees, also leading to a significant speed-up in the CM method \cite{sutherland2012crt}.
\end{enumerate}
These last two tricks only work when the class number is composite. We expect both of them to work
well for generalized class polynomials,
so will mainly restrict to the case of prime class number in our examples,
as this more clearly illustrates the role of the parameter $r(\curve)$.
\end{remark}
The goal of the rest of this section is to show under some hypotheses
that the reduction factor $r(\curve)$ is indeed an asymptotic reduction
factor of the size of the polynomials involved.
For that, we will first introduce
the appropriate notions of ``size''.

\subsection{Measures of polynomials and heights of their roots}
For a polynomial $A\in \CC[X]$,
let $|A|_1$ (resp.~$|A|_\infty$) be the sum (resp.~maximum) of the
absolute values of the coefficients of~$A$.
The \emph{Mahler measure} of a
polynomial $A = a\prod_{i=1}^{n} (X-\alpha_i)\in\CC[X]$ is
$$\mathcal{M}(A) = |a| \prod_{i} \max\{1,|\alpha_i|\}.$$
\begin{lemma}\label{lem:polymeasures}
	We have
\begin{align*}
	|A|_\infty \quad &\leq \quad |A|_1 \quad \leq \quad (n+1)|A|_\infty,\\
	\mathcal{M}(A)\quad &\leq \quad |A|_1 \quad \leq \quad 2^n \mathcal{M}(A),
\end{align*}
\begin{align*}
	\left|\log|A|_1 - \log|A|_\infty \right| &\quad \leq\quad  \log(n+1),\\
	\left|\log|A|_\infty - \log(\mathcal{M}(A)) \right| &\quad \leq\quad  n\log(2).
\end{align*}
\end{lemma}
\begin{proof}
The first two inequalities are by definition and the third is
Equation (6) of~\cite{mahler1960}.
For its converse, observe that we have
$|AB|_1 \leq |A|_1|B|_1$, and hence also
$	|A|_1 \leq |a|  \prod_i \max\{2, 2|\alpha_i|\} \leq 2^n\mathcal{M}(A).$
Then take logarithms.
\end{proof}

For an element $\alpha$ in a number field~$L$ of degree~$n$,
we define its \emph{(absolute logarithmic) height} to be
$$h(\alpha) = \frac{1}{n}\sum_{v} \max\{0,\log|\alpha|_v\},$$
where the sum ranges over the Archimedean
and non-Archimedean absolute values, suitably
normalized (that is, those denoted $||\cdot||_v$
in \cite[\S B.1]{hindry-silverman}). If $\alpha$ is a root of an irreducible $A\in\ZZ[X]$ of degree~$n$,
then we have
\begin{equation}\label{eq:mahler}
	\log(\mathcal{M}(A)) = n h(\alpha).
\end{equation}

\begin{remark}\label{rem:bitsize}
Another measure for the complicatedness of $A$ would be its total
bit size, or the sum $s$ of the logarithms of the absolute values of the nonzero coefficients.
We will instead focus on $|A|_\infty$ for the following reasons.

First of all, for computational purposes, it is more useful to look at $p = \deg(A)\cdot \log|A|_\infty$,
as the required precision (or number of primes with the CRT approach) is proportional to $\log|A|_\infty$
and the number of computations to do with that precision is proportional to $\deg(A)$.

Secondly, we get the impression from numerical computations that
$s$ is close to~$p$. For example, the value of $s/p$ is spread out over the interval $(0.75,  0.9)$ for the larger
discriminants in both Section~\ref{ssec:x0119} and Example~\ref{example:otherbasis}.

Finally, it is hard to prove lower bounds on $s$ other than $s\geq \log|A|_\infty$,
as it seems to already be hard to show that a sufficient proportion of coefficients is nonzero.
\end{remark}

\subsection{Proof of the height reduction}

\begin{theorem}\label{thm:reductionhyper}
	Let $C$ be a modular curve over $\QQ$
	and suppose that $C$ is an elliptic curve
	of rank $0$
	with Weierstrass coordinates $x$ and~$y$.
	Suppose that $\tau\in\HH$ ranges over a sequence of imaginary quadratic points
	for which $C$
	yields real generalized class polynomials~$\HC$, and with
	\begin{equation}\label{eq:limithypothesis}
		\frac{h(j(\tau))}{\log(\log(\#\Cl(\mathcal{O})))}\rightarrow \infty.
	\end{equation}
	Scale each $\F$ such that it has coprime coefficients in~$\ZZ$.
	Then
	$$
	d\cdot \frac{\log|\F|_\infty}{\log|\Fj|_\infty}\rightarrow \frac{1}{r(\curve)},$$
where $d$ is the degree of $K_{\mc{O}}$ over $K(\psi(\tau))$.
\end{theorem}

\begin{remark}\label{rem:reasonablehypothesis}
	We argue that the hypothesis \eqref{eq:limithypothesis}
	is very reasonable. Under GRH, we have
	\begin{equation}\label{eq:littlewood}
		\#\Cl(\mc{O}) = O(\sqrt{|D|}\log(\log|D|)),
	\end{equation}
    where $D$ is the discriminant of~$\mc{O}$
    (see \cite[9.Theorem~1 and 11.~on page~371]{littlewood1928}, suitably
    extended to arbitrary~$D$.)
	Moreover,
	\cite[\S 6.2]{enge-morain2002} gives the approximation
	$\log|\Fj|_\infty\approx \pi\sqrt{|D|} S(D)$,
	with $S(D) = \sum_Q a^{-1}$,
	where the sum ranges over
	reduced primitive quadratic forms $Q = ax^2+bxz+cz^2$
	of discriminant~$D$.
	We now give a heuristic lower bound of this sum on average over all $|D|\leq X$.
	We have $\sum_D S(D) \approx \sum_{Q} a^{-1}$,
	where this time the sum is taken over all reduced
	quadratic forms of negative discriminant $>-X$ (using the heuristic that imprimitive forms have a negligible contribution).
	As we are only computing a lower bound, we may restrict
	to $a \leq \sqrt{X/8}$.
	Then $b$ ranges from $-a$ to $a$,
	and $c$ ranges from $a$ or $a+1$ to $\lfloor(X+b^2)/(4a)\rfloor$;
	a range that contains at least $\lfloor X/(8a)\rfloor$
	integers.
	This yields at least roughly $X/4$ values of $b$ and $c$
	for each $a$, hence
	$\sum_D S(D)$ is roughly at least
	$(X/4)\sum_{a^2\leq X/8} a^{-1} \geq \frac{1}{8} X\log(X)$.
	It follows that the average $S(D)$ is at least proportional to $\log|D|$.
	Thus, for ``average'' $S(D)$, we have that
	$\log|\Fj|_\infty$ is at least proportional to $\sqrt{|D|}\log|D|$.
	Combined with \eqref{eq:littlewood}, \eqref{eq:mahler}, and Lemma~\ref{lem:polymeasures},
	we find for such $D$
	that $h(j(\tau))/\log(\log(\#\Cl(\mc{O})))$ is at least proportional
	to $\log|D| / (\log(\log|D|))^2$.
	We thus see that \eqref{eq:limithypothesis}
	indeed holds for ``average'' $S(D)$.
\end{remark}

Theorem~\ref{thm:reductionhyper} is the analogue of the following
result.

\begin{theorem}[cf.\ Enge-Morain~\cite{enge-morain2002}]
	\label{thm:reductionx}
	Let $f$ be a modular function and suppose that $\tau\in\HH$ ranges
	over a sequence of imaginary quadratic points for which $(f,\tau)$ is a class
	invariant with $h(j(\tau))\rightarrow \infty$.
	Then
	$d\cdot \frac{\log|\Ff|_\infty}{\log|\Fj|_\infty}\rightarrow \frac{1}{r(f)}$,
	where $d$ is the degree of $K_{\mc{O}}$ over $K(f(\tau))$.
\end{theorem}

The goal of the remainder of Section~\ref{sec:reductionfactors}
is to prove Theorem~\ref{thm:reductionhyper}.
We start with a proof of Theorem~\ref{thm:reductionx}.
\begin{proof}
	Let $m$ be the degree of $K(f(\tau))$ over $K$
	and let $n = dm$ be the degree of $K_{\mc{O}}$ over~$K$.
	By Lemma~\ref{lem:polymeasures} and~\eqref{eq:mahler},
	we get $|\frac{1}{n}\log|\Fj|_\infty - h(j(\tau))|\leq \log(2)$
	and $|\frac{d}{n}\log|\Ff|_\infty - h(f(\tau))| \leq \log(2)$.
	
	As $h(j(\tau))\rightarrow \infty$, we also get
	\begin{equation}\label{eq:limit}
	\frac{h(f(\tau))}{h(j(\tau))} \rightarrow \frac{1}{r(f)}
	\end{equation} by \cite[Proposition B.3.5(b)]{hindry-silverman}.
	Altogether, this gives the result.
\end{proof}

\newcommand{\wase}{{d'}}

\begin{proposition}\label{prop:reductionhyper}
	Let $C$ be a modular curve over $\QQ$
	and suppose that $C$ is an elliptic curve
	of rank $0$
	with Weierstrass coordinates $x$ and~$y$.
	For every imaginary quadratic $\tau\in\HH$ for which $C$
	yields a real generalized class polynomial~$\HC$,
	let~$m$ be the degree of $K(\psi(\tau))$ over~$K$ and let $\wase\in\{1,2\}$ be the degree of $K(\psi(\tau))/K(x(\tau))$.
	Scale each $\F$ such that it has coprime coefficients in~$\ZZ$. Then we have 
\begin{equation*}
\left| \log |\F|_\infty - \frac{\wase}{2}\log |\Fx|_\infty\right| < B\max\{1,m\log(\log(m))\},
\end{equation*}
    for some constant $B$ that only depends on $C$ and the choice of Weierstrass model.
\end{proposition}
\begin{proof}
	\textbf{We first put the equation for $C$ in a nice form.}
	We have $C : y^2 + g(x)y = f(x)$.
    Without loss of generality we have $g=0$ and $f\in\ZZ[X]$ monic of odd degree
    such that $f(z)\leq -1$ for all real $z\leq 0$.
    Indeed, we obtain $g=0$ by the substitution $y' = y+\frac{1}{2}g(x)$,
    then do scalings $x' = vx$ and $y'=wy$ to make $f$ integral and (thanks to its odd degree) monic,
    and then do a substitution $x' = x+c$ to make $f(z)\leq -1$ for all $z\leq 0$.
    This affects $\F=A+BY$ and $\Fx$ as follows.
    The first substitution changes $A$ into $A+\frac{1}{2} g(X)B$,
    the second changes $A$ into $A(vX)$ and $B$ into $wB(vX)$,
    and the third changes $A$ into $A(X+c)$.
    Each of these substitutions change $\log(\max\{|A|_1,|B|_1\})$
    at most by~$O(m)$, as does clearing the denominators afterwards.
	
	\textbf{Next, we relate a norm of $\F$ to $\Fx$.} The extra elliptic curve point
	$(a/b^2,c/b^3) := \sum_{\sigma\in G} \sigma(\psi(\tau))\in C(\QQ)$
	from \eqref{eq:defF} (which is minus the Heegner point)
	is torsion by our assumption that 
	$\curve$ has rank~$0$.
	There are finitely many torsion points in $\curve(\QQ)$,
	hence finitely many possibilities for the polynomial
	$T = b^2X-a$.
	Writing $\F = A(X) + B(X)Y$, we get that $N(\F) = A(X)^2 + (-f(X))B(X)^2$
	has the same divisor as the primitive polynomial
	$\Fx^\wase\cdot T$,
	hence there is a constant $s\in\ZZ\setminus \{0\}$
	with $N(\F) = s \Fx^\wase\cdot T$.
	
	We claim that $s = \pm 1$.
	If not, take a prime $p\mid s$ and consider the highest-weight term of
	$(\F\bmod p)$, where $X$ has weight $2$ and $Y$ has weight $\deg(f)$.
	This gives rise to the highest-degree term of $(N(\F)\bmod p)$, which is therefore
	nonzero, a contradiction.
	
	\textbf{Now we use interpolation to bound $\F$ in terms of $\Fx$.}
	We will choose interpolation points $z= g(i) \leq 0$. Note that
	for $z\leq 0$ we have
	$$A(z)^2, B(z)^2 \leq A(z)^2 + (- f(z))B(z)^2 = N(\F)
	\leq \max\{1,|z|\}^m |\Fx|^e_1 |T|_1,$$
	and since there are finitely many polynomials $T$, we get
	$$\log|A(z)|, \log|B(z)| \leq \frac{m}{2} \max\{0,\log|z|\} + \frac{\wase}{2} \log|\Fx|_1 + O(1).$$
	
	Interpolation then gives, for $P \in\{A,B\}$:
	\begin{equation}\label{eq:interpolation}
	P(X) = \sum_{i=1}^{k} P(g(i)) \prod_{j\not=i} \frac{X-g(j)}{g(i)-g(j)},
	\end{equation}
    where $k = \deg(P)+1 = O(m)$.
	
	Taking $g(u) = -\log(eu)^2$, we find $|g(i)-g(j)| \geq |i-j|\min_{z\in[1,k]} |g'(u)|
	= |i-j|\min_{u\in[1,k]} 2\frac{\log(eu)}{u} = 2|i-j|\frac{\log(ek)}{k}$.
	So for each $i$ there are at most $k/\log(k)$ values of $j\not=i$ with $|g(i)-g(j)| < 1$
	and each of them has $|g(i)-g(j)| \geq 1/k$.
	We get 
	$$\log\prod_{j\not=i} \frac{1}{|g(i)-g(j)|}
	\leq (k/\log(k)) \log(k) = k = O(m).$$
For the other factors in \eqref{eq:interpolation},
	we have $\log|X-g(j)|_1 \leq \log(1 + \log(em)^2) = O(\log(\log(m)))$,
	so $\log \prod_{j} |X-g(j)|_1 = O(m\log(\log(m)))$,
	as well as $\log |P(g(i))| \leq \frac{\wase}{2}\log|\Fx|_1 + O(m\log(\log (m)))$.
	Taking the sum in \eqref{eq:interpolation}
	gives another $+\log(k)$, so that the end result is
	$\log |P(X)|_1 \leq \frac{\wase}{2}\log|\Fx|_1 + O(m\log(\log ( m)))$.
	By Lemma~\ref{lem:polymeasures}, this also holds with $|\cdot|_\infty$,
	which proves the upper bound on $\log|\F|_\infty$.
	
	For the lower bound, note that $\Fx^\wase$ is a factor of $Q =A^2-f(X)\cdot B^2$,
	and we have $|Q|_1 \leq |A|_1^2 + |f|_1 |B|_1^2 \leq |f|_1(m+1)^2|\F|_\infty^2$.
	Using the fact that $\mathcal{M}$ is multiplicative by definition
	and is related to $|\cdot|_1$ and $|\cdot|_\infty$ by Lemma~\ref{lem:polymeasures},
	we get exactly what we need:
	$\wase\log|\Fx|_\infty \leq  \wase\log \mathcal{M}(\Fx) + O(m)
	\leq \log\mathcal{M}(Q) + O(m)
	\leq \log|Q|_1 + O(m) \leq 2\log(|\F|_\infty) + O(m)$.
\end{proof}

\begin{proof}[Proof of Theorem~\ref{thm:reductionhyper}]
Denote again by $n=\#\Cl(\mc{O})$ the degree of $K_{\mc{O}}$ over $K$.
	First we apply Theorem~\ref{thm:reductionx} to~$x$ and get
	$d\wase\frac{\log|\Fx|_\infty}{\log|\Fj|_\infty}\rightarrow \frac{2}{r(C)}$.
	Proposition~\ref{prop:reductionhyper}, together with the hypothesis
	$h(j(\tau))/(n\log(\log(n)))\rightarrow \infty$, gives
	$\frac{1}{\wase}\frac{\log|\F|_\infty}{\log|\Fx|}\rightarrow\frac{1}{2}$
	(as in the proof of Theorem~\ref{thm:reductionx}).
	The product of these two limits gives the result.
\end{proof}

\begin{remark}\label{rem:changemodels}
	Theorem \ref{thm:reductionhyper} states
	that asymptotically the effect of the choice of a model of the curve $\curve$ is negligible,
	as is the effect of replacing $f$ by $2f$ or $f+1$ or any other element of $\QQ(f)$
	in Theorem~\ref{thm:reductionx}.

	However, in practice the error terms can be quite large and depend on these choices.
	For example, if $f$ is integral over $\ZZ[j]$ then $\Ff$ is monic, and
    if $f^{-1}$ is integral over $\ZZ[j]$, then $f$ has zero constant coefficient.
    This can make a difference in practical examples as it forces the coefficients at the beginning and end to be small,
     though this improvement is negligible asymptotically by the theorems.
     See also Remark~\ref{rem:bitsize}.
\end{remark}

\section{\texorpdfstring{Class invariants for $\xnulN$ and $\xplusN$}{Class invariants for X0(N) and X0+(N)}}\label{sec:x0}

In this section we assume that $C$ is a quotient over~$\QQ$ of $\xnulN$;
in other words, $C$ is a smooth, projective, geometrically irreducible
curve over $\QQ$ with function field
consisting only of modular functions for $\Gamma^0(N)$
that have rational $q$-expansion.
We will show how to obtain generalized class functions for
every discriminant $D<0$ that is square
modulo~$4N$ (Section~\ref{ssec:x0n}).

In some cases we get further reductions from class invariants
generating subfields of $K_{\mc{O}}$ or from
real class polynomials (Sections~\ref{ssec:x0nramified}--\ref{ssec:x0nplusramified}).

In Sections~\ref{ssec:x0119}--\ref{ssec:comparison} we study what this means for $\xplus{119}$
and in Section~\ref{sec:searchelliptic} we look for more examples of elliptic curve quotients of~$\xnulN$.

\subsection{\texorpdfstring{Class invariants for $\xnulN$}{Class invariants for X0(N)}}
\label{ssec:x0n}

The following result does not require $C$ to
be an elliptic curve, except that (unless $C$ is an elliptic curve)
one needs to read
the definitions in Section~\ref{sec:general}
for the parts about generalized class polynomials.
\begin{proposition}[based on Schertz~\cite{schertz2002weber}]\label{prop:Nsystem}
Let $C = (C, \psi)$ be a quotient over~$\QQ$ of $\xnulN$
and let $D<0$ be a square modulo~$4N$.

There exist $a,b,c\in\ZZ$ with $a,c>0$,
$b^2-4ac = D$, $N\mid c$, and $\gcd(a,N) = \gcd(a,b,c) = 1$.
Choose such $a, b,c$,
let $\tau\in\HH$ be a root of $aX^2 + bX + c$,
with order
$\mc{O} = \ZZ[a\tau]$, which
has discriminant~$D$.
Then we have $$\psi(\tau)\in C(K_{\mathcal{O}}),$$
thus giving rise to a generalized class polynomial $H_{\tau}[C]$.

The Galois orbit of $\psi(\tau)$ can be computed
as follows.
There exists an \emph{$N$-system}, that is,
there exist $\tau_1,\ldots, \tau_n\in\HH$ such that
$(\tau_i\ZZ+\ZZ)_i$ is a system of representatives of $\Cl(\mc{O})$
and such that $\tau_i$ is a root of $a_iX^2 + b_iX + c_i$
with $\gcd(a_i,N) = \gcd(a_i,b_i,c_i)=1$ and
$b_i\equiv b\ \mathrm{mod}\ 2N$.
Moreover, for any such choice, we have
$$\Gal(K_{\mc{O}}/K)\cdot \psi(\tau) = \{\psi(\tau_i) : i=1,\ldots,n\}.$$ 
\end{proposition}
\begin{proof}
	For the existence of $a,b,c$, take an arbitrary square root $b$
	of $D$ modulo~$4N$, let $a=1$, and $c= (b^2-D)/4$.
	Then the existence of an $N$-system is \cite[Proposition~3]{schertz2002weber}.
	
	For any $f\in \QQ(C)$,
	Theorem~4 of Schertz~\cite{schertz2002weber}
	states $f(\tau) 
	\in K_{\mc{O}}\cup\{\infty\}$ and gives
	the $\Gal(K_{\mc{O}}/K)$-orbit as
	$\{g(N\tau_i) : i\}$,
	under an additional condition on the function $f(1/z)$.
	However, the condition on $f(1/z)$ is not
	needed, as stated in Theorems 3.9 and~4.4 of~\cite{enge-streng}.
    This proves the result.
\end{proof}

\subsection{Real class polynomials from ramification}
\label{ssec:x0nramified}

There are some situations in which we can actually get real
class polynomials, cutting the total required bit size in half.
The first such situation is when all primes dividing $N$
ramify.

\begin{proposition}[based on Enge-Morain~\cite{enge-morain2014}]
	\label{prop:realramified}
	Let $C = (C, \psi)$ be a quotient over~$\QQ$ of $\xnulN$
	and let $D<0$ be a discriminant divisible by $N$ if $N$ is odd
	and by $4N$ if $N$ is even.
	
	There exist $a,b,c\in\ZZ$ with $a,c>0$,
	$N\mid b,c$, $\gcd(a,N)=1$, and $b^2-4ac = D$.
	Choose such $a, b,c$,
	let $\tau\in\HH$ be a root of $aX^2 + bX + c$,
	with order
	$\mc{O} = \ZZ[a\tau]$, which
	has discriminant~$D$.
	
Then the $\Gal(K_{\mc{O}}/K)$-orbit of $\psi(\tau)$
is stable under complex conjugation,
and hence we may take $H_{\tau}[C]\in \QQ[X,Y]$.
\end{proposition}
\begin{proof}
If $D$ is odd, take $b=N$, and if $D$ is even, take $b=0$.
If $N$ is even, then we find $4N\mid b^2-D$.
If $N$ is odd, then we find both
$4\mid b^2-D$ and $N\mid b^2-D$,
hence also $4N\mid b^2-D$.
Let $a = 1$ and $c = (b^2-D)/4$.

The complex conjugate of $\psi(\tau)$
is $\psi(-\overline{\tau})$ by the fact that the $q$-expansion
coefficients are real.
Here $-\overline{\tau}$ is a root of $aX^2 - bX + c$,
and as $N\mid b$, we can choose the $N$-system
in Proposition~\ref{prop:Nsystem} in such a way that
$-\overline{\tau} =\tau_i$ for some~$i$.
This proves the result.
\end{proof}

\subsection{\texorpdfstring{Real class polynomials from $\xplusN$}{Real class polynomials from X0+(N)}}
\label{ssec:x0nplus}

The second situation in which we get real class polynomials
is when working with quotients of $\xplusN$.

\begin{proposition}[based on Theorem~3.4 of Enge-Schertz~\cite{enge-schertz2004}]\label{prop:realplus}
	In the situation of Proposition~\ref{prop:Nsystem},
	suppose furthermore that $C$ is a quotient of $\xplusN$,
	and that $\gcd(c/N,N) = 1$.
	
	Then the $\Gal(K_{\mc{O}}/K)$-orbit of $\psi(\tau)$
	is stable under complex conjugation,
	and hence we may take $H_{\tau}[C]\in \QQ[X,Y]$.
\end{proposition}
\begin{proof}
	The complex conjugate of $\psi(\tau)$
	is $\psi(-\overline{\tau})$ by the fact that the $q$-expansion
	coefficients are real.
	As $\psi$ is invariant under the Fricke-Atkin-Lehner involution,
	this in turn is $\psi(\tau')$
	with $\tau' = N/\overline{\tau}$,
	a root of $(c/N)X^2 + bX + Na$.
	As $c/N$ is coprime to~$N$, we can choose the $N$-system
	in Proposition~\ref{prop:Nsystem} in such a way that
	$\tau' =\tau_i$ for some~$i$.
	This proves the result.
\end{proof}

To use this result, we will need $\gcd(c/N,N)=1$, which can
be achieved most of the time, as follows.
\begin{lemma}\label{lem:exceptions}
If $D$ is a square modulo~$4N$ and
$D = F^2D_0$ for a negative fundamental discriminant $D_0$
and a positive integer $F$ coprime to~$N$, then
there exist $a,b,c$ as in Proposition~\ref{prop:Nsystem}
with $\gcd(c/N,N) = 1$.

	More generally, let $D<0$ be a square modulo $4N$.
	Then there exist $a,b,c$ as in Proposition~\ref{prop:Nsystem}
	with $\gcd(c/N,N) = 1$
	if and only if all of the following do not hold.
	\begin{enumerate}
		\item there exists a prime $p\mid N$ with $\ord_p(N)$ odd and $\ord_p(D) > \ord_p(4N)$,
		\item $m:=\ord_2(N) > 0$ and $D$ is of the form $2^{m+1}d$
		with
		$d\equiv 1\ (\mathrm{mod}\ 4)$,
		\item $m:=\ord_2(N) > 0$ and $D$ is of the form $2^{m}d$
		with
		$d\equiv 1\ (\mathrm{mod}\ 8)$.
	\end{enumerate}
\end{lemma}
\begin{proof}
The triple $(a,b,c)$ exists if and only if there exists $b\in\ZZ$
such that for all $p\mid N$: $\ord_p(b^2-D) = \ord_p(4N)$.

Suppose that we are not in case (1), (2), or (3).
By the Chinese remainder theorem, it suffices to find one $b\in\ZZ$
for each $p\mid N$.
So let $p\mid N$ be prime and let $k = \ord_p(4N)$ and $l = \ord_p(D)$.
If $k < l$, then as we are not in case (2), we find that $k$ is even,
and we can take $b = p^{(k/2)}$.
If $k=l$, then we can take $b = p^e$ with $e>k/2$.
Now the case $k>l$ remains.
As $D$ is a square modulo $4N$, there exists
$b_0\in\ZZ$ be such that $D\equiv b_0^2\ (\mathrm{mod}\ 4N)$.
If $\ord_p(b_0^2-D) = \ord_p(4N)$, then we are done,
so suppose $\ord_p(b_0^2-D)>k$.

Note that $2\ord_p(b_0) = l$, hence $l$ is even.
Let $b = b_0 + p^{e}$ with $e$ to be determined later.
We get $b^2-D = (b_0^2-D) + 2p^{e}b_0 + p^{2e}$,
and the terms have valuation $>k$, $e+(l/2)+\ord_p(2)$, $2e$ respectively.

If $p\not=2$, then we choose $e=k-(l/2)$, so $2e = k+(k-l) > k$,
hence $\ord_p(b^2-D) = k$.
If $p=2$ and $k > l+2$, then we choose $e=k-(l/2)-1$, so $2e = k+(k-l-2)>k$,
hence $\ord_p(b^2-D) = k$.

Now only the case $p=2$ with $k-l\in\{1,2\}$ remains.
Write $d = 2^{-l}D$ and $b_1 = 2^{-(l/2)} b_0$, so $b_1$ is odd
and $b_1^2-d$ is divisible by $2^{k-l}$.

In the case $k-l = 1$, we get $b_1^2 - d \equiv 0\pmod{2}$,
and we claim that this is nonzero modulo~$4$.
Indeed, $b_1^2$ is $1$ modulo $4$ and $d$ is not (as we are not in case~(2)).
Therefore $\ord_2(b_1^2-d) = 1$
and $\ord_2(b_0^2-D) = 1+l=k$, so we take $b = b_0$.

In the case $k-l = 2$, we get $b_1^2 - d\equiv 0\pmod{4}$,
and we claim that this is nonzero modulo~$8$.
Indeed, $b_1^2$ is $1$ modulo $8$, and $d$ is not (as we are not in case~(3)).
Therefore $\ord_2(b_1^2-d) = 2$ and $\ord_2(b_0^2-D) = 2 + l = k$, so we take $b = b_0$.

Conversely, suppose that $b$ exists.

In case (1), we have $\ord_p(D)> \ord_p(4N)$, hence
$2\ord_p(b) =\ord_p(4N)$ is odd, contradiction.

In case (2), we have $\ord_2(b^2-2^{m+1}d) = m+2$,
hence $m+1 = 2\ord_2(b)=: 2e$. Write $b = 2^{e}b_1$ and note
$\ord_2(b_1^2-d) = 1$, but $b_1^2-d$ is $0$ modulo $4$.

In case (3), we have $\ord_2(b^2-2^{m}d) = m+2$,
hence $m = 2\ord_2(b)=:2e$. Write $b = 2^eb_1$
and note
$\ord_2(b_1^2-d) = 2$, but $b_1^2-d$ is $0$ modulo $8$.

It remains only to prove the first statement, for which it suffices to show that the exceptions
(1), (2), and (3) all imply $\gcd(N,F)>1$.
In case~(1), we see that $p^2\mid D$ and if $p=2$, then $p^4\mid D$,
hence $p\mid F$.
In cases (2) and~(3), write $D = 2^v d$ with $v\in\{m,m+1\}$.
As $D$ is a square modulo $2^{m+2}$,
we find that $v$ is even, and hence $D = (2^{v/2})^2 d$ for a discriminant~$d$,
so $2\mid F$.
\end{proof}

\begin{lemma}\label{lem:proportions}
	Let $N$ be the product of distinct odd primes
	$p_1,\ldots, p_k$.
	The negative discriminants that are a square modulo $4N$
	and not in one of the exceptions of Lemma~\ref{lem:exceptions}
	have density 
	$$\prod_{i=1}^k \frac{p_i^2 + p_i-2}{2p_i^2}
	$$
	in the set of all negative discriminants.
	
	The negative fundamental discriminants that are a square modulo $4N$
	(which are not in one of the exceptions of Lemma~\ref{lem:exceptions})
	have density 
	$$\prod_{i=1}^k \frac{p_i^2 + p_i-2}{2(p_i^2-1)}$$
	in the set of all fundamental negative discriminants.
\end{lemma}
\begin{proof}
	Being a discriminant is the condition of being $0$ or $1$ modulo~$4$.
	It is equivalent to being a square modulo~$4$.
	This is independent
	of being a square modulo~$p_i$ that does not suffer from~(1),
	which is happens for the $(p_i-1)/2$ residue classes modulo $p_i$
	that are nonzero squares modulo~$p_i$,
	and the $p_i-1$ nonzero residue classes modulo $p_i^2$ that are zero modulo~$p_i$.
	As $p_i(p_i-1)/2 + p_i-1 = (p_i^2+p_i-2)/2$, we get the first statement.
	
	Being a fundamental discriminant means being nonzero
	modulo the squares of all odd primes and being $1, 5, 8, 9, 12, 13$
	modulo $16$.
	This happens for $\zeta(2)^{-1} (1-1/4)^{-1} \frac{6}{16}$ of all negative integers.
	In order to restrict this to products that satisfy the conditions
	of Lemma~\ref{lem:exceptions}, we have to adjust
	the Euler product exactly by the given factor.
\end{proof}
For example, if $N = 119 = 7\cdot 17$, then the numbers in
Lemma~\ref{lem:proportions} are $> 0.2898$
and $19/64 > 0.2968$.

\subsection{Lower-degree class polynomials from ramification}
\label{ssec:x0nplusramified}

In the case where all primes dividing $N$ ramify, we get an even
greater size reduction. The point $\psi(\tau)$
will then be defined over a subfield, cutting the degree of its
minimal polynomial in half. This in turn also cuts the height
of the coefficients of this polynomial in half,
as we get $d\geq 2$ in Theorem~\ref{thm:reductionhyper}.
The amount of work required for computing the class polynomial,
as well as the bit size of the polynomial (Remark~\ref{rem:bitsize}), is related to
the degree times the logarithm of the largest coefficient,
and this product is reduced by a factor $\geq 2\times 2 \times r(C)= 4r(C)$.

\begin{proposition}[based on Enge-Schertz~\cite{enge-schertz2013}]
	\label{prop:ramifiedsubfield}
	Let $C = (C, \psi)$ be a quotient over~$\QQ$ of $\xnulN$
	and let $D=F^2D_0<0$ be such that $N\mid D$,
	$\gcd(F,N)=1$, and $D\not\in\{N,4N\}$.
	
	There exist $a,b,c\in\ZZ$ with $a>0$,
	$N\mid b$, $c=N$, $b^2-4ac = D$,
	and $\gcd(a,b,c) = 1$.
	Choose such $a, b,c$,
	let $\tau\in\HH$ be a root of $aX^2 + bX + c$,
	with order
	$\mc{O} = \ZZ[a\tau]$, which
	has discriminant~$D$.
	
	Let $\mathfrak{n} = ((-b+\sqrt{D})/2, a)$,
	and let $K_{\mc{O}}^{[\mathfrak{n}]}$ be the subfield
	of $K_{\mc{O}}$ fixed by the image of $\mathfrak{n}$ under the
	Artin map.
	Then $[\mathfrak{n}]$ has order $2$ in $\Cl(\mc{O})$
	and $\psi(\tau)\in C(K_{\mc{O}}^{[\mathfrak{n}]})$,
	where $K_{\mc{O}}$ has degree $2$ over
	$K_{\mc{O}}^{[\mathfrak{n}]}$.

	We get $m \leq \#\Cl(\mc{O})/2$ in the definition of~$\F$,
	we get $\F\in\QQ[X,Y]$, and we get
	and $d\geq 2$ in Theorem~\ref{thm:reductionhyper}.
	
	If $\mathfrak{a}_i$ are the ideals $\tau_i\ZZ+\ZZ$
	of an $N$-system, then $\fraka_i$ 
	and $\fraka_i\mathfrak{n}$ yield the same point
	$\psi(\tau_i)$, while $\fraka_i^{-1}$ 
	and $\fraka_i^{-1}\mathfrak{n}$ yield
	$\overline{\psi(\tau_i)}$.
\end{proposition}
\begin{proof}
	This is exactly what we get when applying
	\cite[Theorem~9]{enge-schertz2013}
	to the coordinate functions $f$ of~$C$.
\end{proof}

\subsection{\texorpdfstring{Numerical results for $\xplus{119}$}{Numerical results for X0+(119)}}\label{ssec:x0119}

For the rest of this section we will return to our main Example \ref{example:main}, so set $N=119=7\cdot 17$.
For any $\tau$ as in Proposition~\ref{prop:realplus}, we have
$H_{\tau}[C]\in \QQ[X,Y]$.
By scaling, we may assume that the coefficients of $H_{\tau}[C]$ are integral and coprime,
and that the leading coefficient
(i.e.\ the coefficient of the monomial of highest degree as a function on $C$)
is positive,
and this uniquely determines $H_{\tau}[C]\in\ZZ[X,Y]$.

For any discriminant $D<0$ coprime to $N$ such that $D$ is a square modulo $N$,
there are two generalized class polynomials (depending on the choice of $\tau$).
We experimentally computed both of these for all fundamental discriminants of prime class number $<100$.
The main reason for restricting to prime class number is to exclude the two tricks of Remark \ref{rmk:besidesrc};
for these discriminants,
the reduction factor thus provides a fair comparison
with the Hilbert class polynomial.
The method we employ numerically evaluates class invariants by their $q$-expansions, and finds a minimal polynomial relation \eqref{eq:lineareqn} using lattice basis reduction (LLL).
We leave faster methods for future research, but see Section~\ref{sec:runtime} for the first ideas.
Since the $q$-expansions can only be evaluated up to finite precision, this does not result in provably correct polynomials,
although -- based on heuristic estimates -- they are highly unlikely to be incorrect.

A few examples of computed polynomials are listed in Table \ref{tab:119}.
Here, for the given discriminant $D$,
we consistently chose $\tau$ such that its primitive equation is $X^2+bX+(b^2-D)/4$
with $b\in\ZZ_{>0}$ \emph{minimal} satisfying $b^2\equiv D\pmod{4N}$
and $\gcd((b^2-D)/(4N) ,N) = 1$.

\begin{table}[htbp]
\setlength{\tabcolsep}{3pt}
\renewcommand{\arraystretch}{1.2}
\begin{tabular}{c|c|c}
 $D$  & $n$ & $F_{\tau}[C]$ \\ \hline
 $-52$ & $2$ & $y+1$ \\ \hline
 $-523$ & $5$ & $x^3 + x^2 - 2xy - 3x - 2y$\\ \hline
 \multirow{2}{*}{$-5347$} & \multirow{2}{*}{$13$} & $x^7 + 58 x^6 - 13 x^5 y - 39 x^5 - 143 x^4 y - 85 x^4 - 135 x^3 y$\\
 & &  $- 19 x^3 - 51 x^2 y + 47 x^2 + 7 x y - 12 x - y + 1$ \\ \hline
 \multirow{5}{*}{$-15139$} & \multirow{5}{*}{$29$} & $x^{15} + 1028x^{14} - 40x^{13}y + 37342x^{13} - 10557x^{12}y + 79865x^{12}$\\
 & & $  - 167759x^{11}y - 385199x^{11} - 474165x^{10}y - 425857x^{10} - 69261x^9y$\\
 & &   $+345059x^9 + 493309x^8y + 309689x^8 + 168403x^7y - 132377x^7 $\\
 & & $- 145439x^6y - 22165x^6 - 16029x^5y + 16139x^5 + 15225x^4y - 4867x^4$\\
 & & $ - 7127x^3y - 456x^3 + 623x^2y + 423x^2 + 337xy - 65x - 64y$\\ \hline
\end{tabular}
\\[2mm]
\caption{Some conjecturally correct generalized class functions for $C=\xplus{119}$. The second column lists the class number $n$ of the discriminant $D$.}
\label{tab:119}
\end{table}

Still assuming that $H_{\tau}[C]$ is scaled such that it has coprime coefficients in $\ZZ$, we denote by
\begin{equation*}
r_A(\tau):=\frac{\log|H_{\tau}[j]|_{\infty}}{\log|H_{\tau}[C]|_{\infty}}
\end{equation*}
the \emph{practical}
\emph{reduction factor} of $\tau$.
Under the assumption $h(j(\tau))/\log(\log(n))\to\infty$
for $n=\#\Cl(\mc{O})$
(cf.\ Theorem \ref{thm:reductionhyper}) we have $d^{-1}r_A(\tau)\to r(C)$.
Experimentally obtained practical reduction factors, plotted against both the class number $n$ and
$\log(|H_{\tau}[j]|_{\infty})/\log(\log(n))$,
can be seen in Figure \ref{fig:119h100}. To visualize the role of the class number and
the hypothesis $h(j(\tau))/\log(\log(n))\to\infty$,
the points of higher class number are given a darker color in the second figure. 

\begin{figure}[ht]
\begin{subfigure}[ht]{1\textwidth}
\centering
\includegraphics[scale=0.8]{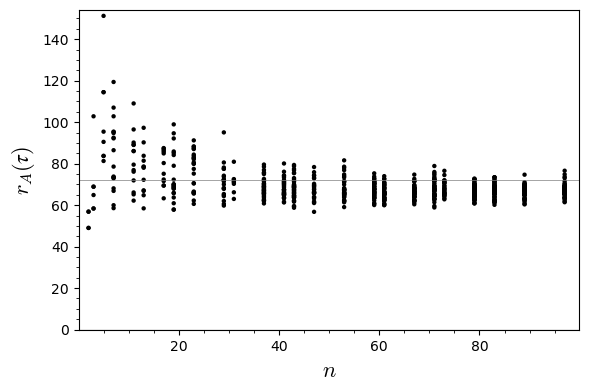}
\end{subfigure}
\vfill
\begin{subfigure}[ht]{1\textwidth}
\centering
\includegraphics[scale=0.8]{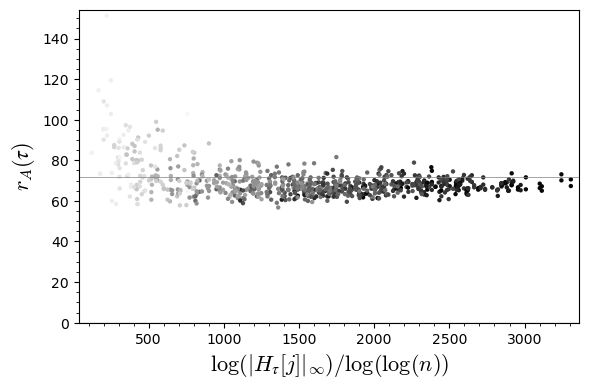}
\end{subfigure}
\caption{Practical reduction factors for $H_{\tau}[\xplus{119}]$ for fundamental discriminants $D$ with $\gcd(D,N)=1$ and prime class number $n<100$.}
\label{fig:119h100}
\end{figure}

The values of the practical reduction factor $r_A(\tau)$ seem to be around their expected asymptotic value $r(C)=72$ (represented by the horizontal grey line), though the convergence is not apparent; especially compared to, e.g.\ some classical class polynomials \cite[Fig.\ 1]{enge-morain2002}. However, in practical applications (see Section \ref{sec:apply}), the class numbers employed are typically several orders of magnitude higher (cf.\ e.g.\ \cite{sutherland2012crt}), so here we expect the speed of convergence not to cause major deviations in expected running times (cf.\ Section \ref{sec:runtime}). For small class numbers, one can in practice even take advantage of this phenomenon by constructing generalized class polynomial with surprisingly good practical reduction factors by selecting a basis of $\mc{L}(\infty\mc{D})$ different from $1,x,y,x^2,xy,\ldots$ (see Example \ref{example:otherbasis}).

\subsection{Comparison with existing class invariants}\label{ssec:comparison}
Real class invariants typically arise subject to congruence conditions on the discriminant.
For example, Weber's functions with reduction factor $72$ are not known to give class invariants for discriminants $\equiv 5\pmod{8}$.
The reduction factors obtained by class invariants coming from the family of (double) eta quotients $\mathfrak{w}_n$ and $\mathfrak{w}_{p,q}$
(such as the Weber function~$\mathfrak{w}_2$, as well as the function $\mathfrak{w}_{7,17}$ of Example \ref{example:main})
have been extensively studied; cf.\ most notably \cite{enge-morain2014}.
These modular functions are not known to yield
class invariants if $D$ is not a square modulo $4n$ or $4pq$.
Hence, to the best of our knowledge, they also are not applicable
to discriminants $\equiv 5\pmod{8}$ as soon as $n$, $p$ or $q$ is even.
Excluding these cases, the (double) eta quotient with highest known reduction factor is $\mathfrak{w}_9$,
with a reduction
factor of $36$ \cite[Table 7.1]{enge-morain2014}.

A less-studied generalization are \emph{multiple eta quotients} \cite{enge-schertz2013},
which are quotients of products of $2^k$ eta functions.
As far as we know these do not yield reduction factors better than $36$ for $k>1$.

The only other known family of ``good'' class invariants
(in the sense that they have large reduction factors)
are the Atkin functions $A_p$
for prime numbers~$p$,
defined to be the smallest-degree functions
in $\mathcal{L}(\infty \D)$,
where $\D$ is the unique cusp of $\xplus{p}$.
The ``best'' known one here is $A_{71}$,
again with a reduction factor of $36$,
owing to the fact
that $\xplus{71}$ has genus zero~\cite[\S3]{enge-sutherland2010}).

The curve $C=\xplus{119}$
has a reduction factor $r(C) = 72$
and yields real class invariants
whenever $D$ is a square modulo $4\cdot7\cdot 17$
and not divisible by $7^2$ or $17^2$.
The set of such $D$ has density $> 28.98\%$
among the set of all negative discriminants
(by Lemma~\ref{lem:proportions}).
Out of these discriminants, one-fourth are $\equiv 5\pmod{8}$.
Hence, for at least $28.98\%\cdot\frac{1}{4}> 7.24\%$
of imaginary quadratic discriminants,
the reduction factor
exceeds the previously best known reduction factors
by a factor of at least two.

\begin{remark}
One should note that the above comparison does not
take into account the discussion of Remark \ref{rmk:besidesrc}.
Most importantly, the reduction factor is \emph{not} synonymous
with the true size reduction of the class polynomials.
Indeed, as noted in that remark,
the record-breaking CM construction \cite{sutherland2012crt}
uses the Atkin invariant $A_{71}$ of reduction factor $36$,
because the effective size reduction of class polynomials
is by a factor of roughly $2^2\cdot 36 = 144$ for certain discriminants.
However,
by Section~\ref{ssec:x0nplusramified},
the same trick applies to generalized class polynomials,
leading for $\xplus{119}$ to a size reduction of $2^2\cdot 72 = 288$,
again for a positive density subset of discriminants.
In Figure \ref{fig:bit} we plot the practical reductions in bit size
we found compared to the Hilbert class polynomial using this trick.

\begin{figure}[ht]
\centering
\includegraphics[scale=0.8]{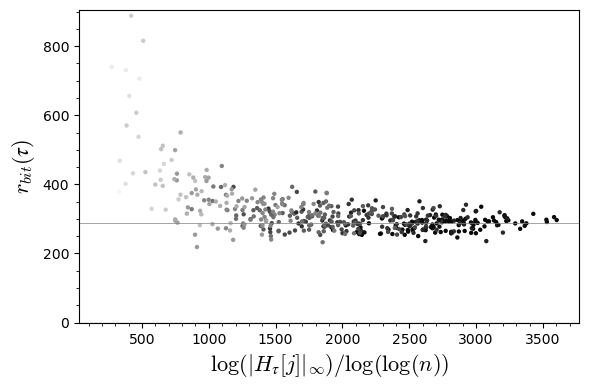}
\caption{Bit-length reduction for $H_{\tau}[\xplus{119}]$ for discriminants $D\equiv 0\pmod{119}$ of class number $n<100$.}
\label{fig:bit}
\end{figure}

\end{remark}

\begin{remark}
	Note that the ``classical'' class polynomial $\Hx$,
	arising from the function $x$ on $\xplus{119}$
	by itself attains a reduction factor of $36$ for the same $28.98\%$ of discriminants.
	This beats all previously-known class invariants for
	a smaller subset ($\approx 1.2\%$) of discriminants: those that additionally are
	non-square modulo both $3$ and~$71$.
	This $x$ can be viewed as a generalisation of the Atkin functions to non-prime levels:
	it is the function of minimal degree in $\mathcal{L}(\infty\D)$ for one
	of the cusps $\D$ of $\xplus{119}$.
	
	Similarly, the degree-two map of the hyperelliptic curve $\xplus{191}$ (not to be confused with $119$) has reduction factor~$48$,
	as observed by David Kohel in the AGC${}^2$T 2021 Zulip group chat.
	This beats the reduction factor~$32$ of the Atkin function $A_{191}$ of degree $3$ on the same curve (see Example~\ref{example:hyperelliptic}).
	
	This shows that the search for generalized class invariants can even uncover new
	``classical'' class invariants.
\end{remark}

\subsection{More modular curves of genus one}\label{sec:searchelliptic}

We searched for more elliptic curves that could be used,
and the results are in Tables \ref{tab:ellipticx0plus}, \ref{tab:ellipticx0}, and~\ref{tab:cremona}.
In our search, we used the fact that $X_0(N)$ is well-studied
and that there is an isomorphism $X_0(N)\rightarrow \xnulN : z\mapsto Nz$.
Surpisingly, we found lots of elliptic curves with
reduction factor $72$ and no elliptic curves
with a greater reduction factor.

In Section~\ref{sec:general}, we will allow
curves of higher genus, which do achieve arbitrarily high
values of~$r(C)$.
Moreover, our search is by no means exhaustive,
as Tables \ref{tab:ellipticx0plus} and~\ref{tab:ellipticx0}
restrict to maps $\phi : X\rightarrow C$ of degree $\leq 2$
and Table~\ref{tab:cremona} only looks at one curve
$X = \xnulN$ per isomorphism class of curves~$C$.
For example, the curve $C = \xplus{119}$ has $r(C) = 72$.
However, in the Cremona database, it is listed as 17a4,
and comes with a modular parametrization
$\phi_{17} : X_0(17)\rightarrow C$ of degree~$1$,
which has $r(\phi_{17}) = 18$. This is why $C$
does not appear in Table~\ref{tab:cremona}.

Finally, the tables are restricted to quotients of~$\xnulN$.
Letting go of $\xnulN$, we find that the genus-one modular curves
$7C^1$, 
$8K^1$, 
$9H^1$, 
$12V^1$, 
$15I^1 = X_1(15)$, 
$16M^1$, 
$24J^1$, 
$27C^1$, 
$32E^1$ 
in the Pauli-Cummins database~\cite{pauli-cummins2003} all achieve  $r(C) \in\{84, 96, 108\}$.
We have not pursued these curves yet, as Proposition~\ref{prop:Nsystem} does not apply to them.

\begin{table}[ht]
	\[
	\begin{array}{|rl|c|r|r|c|r|}
		\hline
		N & & g(X) & r(X) & \deg(\phi) & g(C) & r(C)\\
		\hline
		119&= 7 \cdot 17 & 1 & 72 & 1 & 1 & 72 \\
		120&= 2^{3} \cdot 3 \cdot 5 & 7 & 144 & 2 & 1 & 72 \\
		144&= 2^{4} \cdot 3^{2} & 5 & 144 & 2 & 1 & 72 \\
		176&= 2^{4} \cdot 11 & 7 & 144 & 2 & 1 & 72 \\
		188&= 2^{2} \cdot 47 & 9 & 144 & 2 & 1 & 72 \\
		131&= 131 & 1 & 66 & 1 & 1 & 66 \\
		75&= 3 \cdot 5^{2} & 1 & 60 & 1 & 1 & 60 \\
		95&= 5 \cdot 19 & 1 & 60 & 1 & 1 & 60 \\
		171&= 3^{2} \cdot 19 & 5 & 120 & 2 & 1 & 60 \\
		54&= 2 \cdot 3^{3} & 1 & 54 & 1 & 1 & 54 \\
		81&= 3^{4} & 1 & 54 & 1 & 1 & 54 \\
		90&= 2 \cdot 3^{2} \cdot 5 & 4 & 108 & 2 & 1 & 54 \\
		108&= 2^{2} \cdot 3^{3} & 4 & 108 & 2 & 1 & 54 \\
		110&= 2 \cdot 5 \cdot 11 & 5 & 108 & 2 & 1 & 54 \\
		135&= 3^{3} \cdot 5 & 4 & 108 & 2 & 1 & 54 \\
		136&= 2^{3} \cdot 17 & 6 & 108 & 2 & 1 & 54 \\
		142&= 2 \cdot 71 & 8 & 108 & 2 & 1 & 54 \\
		159&= 3 \cdot 53 & 4 & 108 & 2 & 1 & 54 \\
		101&= 101 & 1 & 51 & 1 & 1 & 51 \\
		48&= 2^{4} \cdot 3 & 1 & 48 & 1 & 1 & 48 \\
		56&= 2^{3} \cdot 7 & 1 & 48 & 1 & 1 & 48 \\
		63&= 3^{2} \cdot 7 & 1 & 48 & 1 & 1 & 48 \\
		64&= 2^{6} & 1 & 48 & 1 & 1 & 48 \\
		84&= 2^{2} \cdot 3 \cdot 7 & 4 & 96 & 2 & 1 & 48 \\
		96&= 2^{5} \cdot 3 & 3 & 96 & 2 & 1 & 48 \\
		105&= 3 \cdot 5 \cdot 7 & 5 & 96 & 2 & 1 & 48 \\
		124&= 2^{2} \cdot 31 & 6 & 96 & 2 & 1 & 48 \\
		128&= 2^{7} & 3 & 96 & 2 & 1 & 48 \\
		141&= 3 \cdot 47 & 6 & 96 & 2 & 1 & 48 \\
		155&= 5 \cdot 31 & 4 & 96 & 2 & 1 & 48 \\
		191&= 191 & 2 & 96 & 2 & 0 & 48 \\
\hline
	\end{array}\]
	\caption{The curves $X = \xplusN$
	for which there exists a map $\phi : X \rightarrow C$
	of degree $\leq 2$
	with $g(C) \leq 1$
	and $r(C) \geq 48$.
	We used Furumoto-Hasegawa~\cite{furumoto-hasegawa1999}
	and Jeon~\cite
{jeon2018}
	to get a complete list.
    }\label{tab:ellipticx0plus}
\end{table}

\begin{table}[ht]
	\[
	\begin{array}{|rl|c|r|r|c|r|}
		\hline
		N & & g(X) & r(X) & \deg(\phi) & g(C) & r(C)\\
		\hline
36&= 2^{2} \cdot 3^{2} & 1 & 72 & 1 & 1 & 72 \\
60&= 2^{2} \cdot 3 \cdot 5 & 7 & 144 & 2 & 1 & 72 \\
72&= 2^{3} \cdot 3^{2} & 5 & 144 & 2 & 1 & 72 \\
92&= 2^{2} \cdot 23 & 10 & 144 & 2 & 1 & 72 \\
94&= 2 \cdot 47 & 11 & 144 & 2 & 1 & 72 \\
49&= 7^{2} & 1 & 56 & 1 & 1 & 56 \\
24&= 2^{3} \cdot 3 & 1 & 48 & 1 & 1 & 48 \\
32&= 2^{5} & 1 & 48 & 1 & 1 & 48 \\
42&= 2 \cdot 3 \cdot 7 & 5 & 96 & 2 & 1 & 48 \\
48&= 2^{4} \cdot 3 & 3 & 96 & 2 & 0 & 48 \\
62&= 2 \cdot 31 & 7 & 96 & 2 & 1 & 48 \\
69&= 3 \cdot 23 & 7 & 96 & 2 & 1 & 48 \\
		\hline
	\end{array}\]
	\caption{The curves $X = \xnulN$
		for which there exists a map $\phi : X \rightarrow C$
		of degree $\leq 2$
		with $g(C) \leq 1$
		and $r(C) \geq 48$
		and $N$ is not already in Table~\ref{tab:ellipticx0plus}.
		We used Ogg~\cite{ogg1974}
		and Bars~\cite{bars1999}
		to get a complete list.
	}\label{tab:ellipticx0}
\end{table}

\begin{table}[ht]
	\[
	\begin{array}{|r|l|r|r|c|r|}
		\hline
		E & N & r(X) & \deg(\phi) & \mathrm{rank}(E) & r(C)\\
		\hline
\href{http://www.lmfdb.org/EllipticCurve/Q/36a1}{\mathrm{36a1}} & 2^{2} \cdot 3^{2} & 72 & 1 & 0 & 72 \\
\href{http://www.lmfdb.org/EllipticCurve/Q/92a1}{\mathrm{92a1}} & 2^{2} \cdot 23 & 144 & 2 & 0 & 72 \\
\href{http://www.lmfdb.org/EllipticCurve/Q/94a1}{\mathrm{94a1}} & 2 \cdot 47 & 144 & 2 & 0 & 72 \\
\href{http://www.lmfdb.org/EllipticCurve/Q/144a1}{\mathrm{144a1}} & 2^{4} \cdot 3^{2} & 288 & 4 & 0 & 72 \\
\href{http://www.lmfdb.org/EllipticCurve/Q/368e1}{\mathrm{368e1}} & 2^{4} \cdot 23 & 576 & 8 & 1 & 72 \\
\href{http://www.lmfdb.org/EllipticCurve/Q/558a1}{\mathrm{558a1}} & 2 \cdot 3^{2} \cdot 31 & 1152 & 16 & 1 & 72 \\
\href{http://www.lmfdb.org/EllipticCurve/Q/704a1}{\mathrm{704a1}} & 2^{6} \cdot 11 & 1152 & 16 & 1 & 72 \\
\href{http://www.lmfdb.org/EllipticCurve/Q/704k1}{\mathrm{704k1}} & 2^{6} \cdot 11 & 1152 & 16 & 1 & 72 \\
\href{http://www.lmfdb.org/EllipticCurve/Q/1728a1}{\mathrm{1728a1}} & 2^{6} \cdot 3^{3} & 3456 & 48 & 1 & 72 \\
\href{http://www.lmfdb.org/EllipticCurve/Q/1728v1}{\mathrm{1728v1}} & 2^{6} \cdot 3^{3} & 3456 & 48 & 1 & 72 \\
\href{http://www.lmfdb.org/EllipticCurve/Q/3456a1}{\mathrm{3456a1}} & 2^{7} \cdot 3^{3} & 6912 & 96 & 1 & 72 \\
\href{http://www.lmfdb.org/EllipticCurve/Q/3456e1}{\mathrm{3456e1}} & 2^{7} \cdot 3^{3} & 6912 & 96 & 0 & 72 \\
\href{http://www.lmfdb.org/EllipticCurve/Q/131a1}{\mathrm{131a1}} & 131 & 132 & 2 & 1 & 66 \\
\hline
\href{http://www.lmfdb.org/EllipticCurve/Q/575a1}{\mathrm{575a1}} & 5^{2} \cdot 23 & 720 & 12 & 1 & 60 \\
\href{http://www.lmfdb.org/EllipticCurve/Q/711a1}{\mathrm{711a1}} & 3^{2} \cdot 79 & 960 & 16 & 1 & 60 \\
\href{http://www.lmfdb.org/EllipticCurve/Q/755b1}{\mathrm{755b1}} & 5 \cdot 151 & 912 & 16 & 1 & 57 \\
\href{http://www.lmfdb.org/EllipticCurve/Q/999b1}{\mathrm{999b1}} & 3^{3} \cdot 37 & 1368 & 24 & 1 & 57 \\
\href{http://www.lmfdb.org/EllipticCurve/Q/49a1}{\mathrm{49a1}} & 7^{2} & 56 & 1 & 0 & 56 \\
\href{http://www.lmfdb.org/EllipticCurve/Q/1323m1}{\mathrm{1323m1}} & 3^{3} \cdot 7^{2} & 2016 & 36 & 1 & 56 \\
\href{http://www.lmfdb.org/EllipticCurve/Q/243a1}{\mathrm{243a1}} & 3^{5} & 324 & 6 & 1 & 54 \\
\href{http://www.lmfdb.org/EllipticCurve/Q/405c1}{\mathrm{405c1}} & 3^{4} \cdot 5 & 648 & 12 & 1 & 54 \\
\href{http://www.lmfdb.org/EllipticCurve/Q/459a1}{\mathrm{459a1}} & 3^{3} \cdot 17 & 648 & 12 & 1 & 54 \\
\href{http://www.lmfdb.org/EllipticCurve/Q/101a1}{\mathrm{101a1}} & 101 & 102 & 2 & 1 & 51 \\
\href{http://www.lmfdb.org/EllipticCurve/Q/335a1}{\mathrm{335a1}} & 5 \cdot 67 & 408 & 8 & 1 & 51 \\
\href{http://www.lmfdb.org/EllipticCurve/Q/591a1}{\mathrm{591a1}} & 3 \cdot 197 & 792 & 16 & 1 & 99/2 \\
\href{http://www.lmfdb.org/EllipticCurve/Q/485b1}{\mathrm{485b1}} & 5 \cdot 97 & 588 & 12 & 1 & 49 \\
\href{http://www.lmfdb.org/EllipticCurve/Q/723b1}{\mathrm{723b1}} & 3 \cdot 241 & 968 & 20 & 1 & 242/5 \\
\href{http://www.lmfdb.org/EllipticCurve/Q/69a1}{\mathrm{69a1}} & 3 \cdot 23 & 96 & 2 & 0 & 48 \\
\href{http://www.lmfdb.org/EllipticCurve/Q/105a1}{\mathrm{105a1}} & 3 \cdot 5 \cdot 7 & 192 & 4 & 0 & 48 \\
\href{http://www.lmfdb.org/EllipticCurve/Q/141d1}{\mathrm{141d1}} & 3 \cdot 47 & 192 & 4 & 1 & 48 \\
\href{http://www.lmfdb.org/EllipticCurve/Q/155c1}{\mathrm{155c1}} & 5 \cdot 31 & 192 & 4 & 1 & 48 \\
\href{http://www.lmfdb.org/EllipticCurve/Q/213a1}{\mathrm{213a1}} & 3 \cdot 71 & 288 & 6 & 0 & 48 \\
\hline
	\end{array}\]
	\caption{The elliptic curves $E/\QQ$ of conductor $< 500.000$
		such that the modular parametrization
		$\phi : X \rightarrow E$ according to
		the LMFDB~\cite{lmfdb,cremona1992,sage}
		gives $r(C) \geq 66$ or gives $r(C) \geq 48$ and odd~$N$.
	}\label{tab:cremona}
\end{table}

\section{Application: the CM method}\label{sec:apply}
	
Class polynomials are used in the \emph{CM method} for constructing
elliptic curves over finite fields with a specified characteristic polynomial
of Frobenius. 

The input to the CM method is a monic quadratic polynomial $P = x^2 - tx + q\in\ZZ[x]$,
where $q$ is a prime power coprime to $t$, and the discriminant $d=t^2-4q$ is negative.
The output is an elliptic curve $E/\FF_q$ with $q+1-t$ rational points,
which has $P$ as its characteristic polynomial of Frobenius. 

The algorithm of the classical CM method (without using class invariants for now) is as follows. Let $K = \QQ(\sqrt{d})$.
\begin{enumerate}
\item \label{it:computeHCP} Compute the Hilbert class polynomial $H_K$ of $\mathcal{O}_K$.
\item \label{it:findroot}
Find a root $j_0\in\FF_q$ of $H_K$ (which is known to split
into linear factors in~$\FF_q$).
\item \label{it:constructandcheck}
Construct an elliptic curve $E/\FF_q$ with $j(E)=j_0$.
Compute all twists of $E$
and return the one with $q+1-t$ rational points.
\end{enumerate}

In practice, one can discard the curves for which $(q+1-t)Q\not=O$ for some random point $Q$, although there are also more straightforward methods to select the correct twist \cite{rs2010choosing}.

As the degree and height of the Hilbert class polynomial grow quickly with
the absolute value of
the discriminant $\Delta_K$ of~$K$, the CM method is only feasible for small values of~$|\Delta_K|$.
The record computation of \cite{sutherland2012crt} uses class invariants, specifically arising from the Atkin function $A_{71}$.
Combined with the tricks listed in Remark~\ref{rmk:besidesrc} this allows to handle a case where $|\Delta_K| > 10^{16}$.

We will now describe how to apply the CM method using generalized class polynomials.
Hence let $C$ be an elliptic modular curve. Since we are working with alternative class invariants instead of the usual $j$-invariant,
we will relate the two using \emph{modular polynomials}
as follows.

\begin{lemma}\label{lem:modularpolynomials} 
Let $d_j:=[\QQ(C,j):\QQ(C)]$. Then there exists a polynomial
$\Psi_C=\sum_{i=0}^{d_j} f_iZ^i\in \ZZ[X,Y][Z]$
of degree $d_j$ in $Z$ such that
\begin{enumerate}[(i)]
\item $\Psi_C(j)=0$;
\item \label{it:degreey} $\deg_{Y}(f_i) \leq 1$ for each $i$;
\item \label{it:coprime} the coefficients (in $\ZZ$) of $\Psi_C$ viewed as an element of $\ZZ[X,Y,Z]$ are coprime;
\item \label{it:onezero} viewed as elements of $\QQ(C)$, the $f_i$ have at most one common zero in $C(\overline{\QQ})$.
\end{enumerate}
Furthermore, $\Psi_C$ is unique up to sign.
\end{lemma}
\begin{proof}
Consider the minimal polynomial 
$\Psi_C^0 = \sum_{i=0}^{d_j} g_iZ^i\in \QQ(C)[Z]$
of $j$ over $\QQ(C)$.
Let
\begin{equation*}
\mc{E}:=\sum_{P\in C\setminus\{O\}} \min_i(\ord_P(g_i))(P).
\end{equation*}
Then $\mc{E} - \left(\sum_{P\in C}\ord_P(\mc{E})P\right) - (\deg(\mc{E})-1)(O)$ is a
$\QQ$-rational principal divisor.
There is a unique function $g$ up to $\QQ^\times$-scaling
such that $\div(g)=\mc{E}$.
Dividing each $g_i$ by $g$ gives $g_i\in \calL(\infty(O)) = \QQ[x,y]$
satisfying~\eqref{it:onezero}
and unique up to~$\QQ^\times$.
Now take representatives $f_i$ satisfying~\eqref{it:degreey}
and scale to get \eqref{it:coprime},
which makes $\Psi_C$ unique up to sign.
\end{proof}
For each curve $C$ with which we would like to apply the
generalized CM method, the polynomial
$\Psi_C\in\ZZ[X,Y,Z]$ can be precomputed and stored.
Next we need a criterion for which discriminants $D$ yields class invariants.
For example, if $C=\xplusN$ then this is given by Proposition~\ref{prop:Nsystem}.
Now, given a desired characteristic polynomial of Frobenius $x^2-tx+q$ such that $D=t^2-4q$ satisfies this criterion, we have the following algorithm for sufficiently large~$|D|$.

\begin{enumerate}
\item[(1)]\label{it:computeF} Compute a generalized class function $F$
of discriminant $D$ as well as its Heegner point $Q$.
\item[(2a)] 
Find a zero $P=(x,y)\in C(\FF_q)$ of $F$ that is 
neither~$-Q$ nor a common root of the polynomials
$f_1,\ldots, f_{d_j}$ of Lemma~\ref{lem:modularpolynomials}.
\item[(2b)] Find all roots $j_0\in\FF_q$ of the polynomial $\Psi_C(x,y,Z)\in\FF_q[Z]$.
\item[(3)] For each $j_0$,
construct an elliptic curve $E/\FF_q$ with $j(E)=j_0$
and all of its twists up to isomorphism over~$\FF_q$.
Return one with $q+1-t$ rational points.
\end{enumerate}

The main advantage compared to the classical CM method,
both in terms of memory and speed, is expected to be
in the (dominant) first step \eqref{it:computeHCP}
(see Section~\ref{sec:runtime}).
Out of the computationally non-dominant steps,
only (2a) is less straightforward. One way to proceed would be as follows.

\begin{enumerate}[(i)]
\item\label{it:norm} Compute $F_x:=N_{\FF_q(\curve)/\FF_q(x)}(F)$.
\item Find a root $x\in\FF_q$ of $F_x$.
\item Solve for the corresponding value of $y$
using the linear polynomial $\HC(x,Y)$, 
or continue with both solutions $y$ coming from the Weierstrass equation.
\end{enumerate}

\begin{remark}
The polynomial $F_x$ is very close to the classical class polynomial $H_{\tau}[x]$;
indeed, it has the same roots, together with one additional root at the
$x$-coordinate of the Heegner point of~$F$.
The norm computation in step \eqref{it:norm} is however computationally
asymptotically dominated by the computation of~$F$.
\end{remark}

    \section{The computational benefits of our invariants}\label{sec:runtime}
    
    \subsection{Space complexity of the functions}
    
    The advantage of using generalized class functions lies in their size.
    This already gives a serious advantage when storing
    one or more class polynomials for later use,
    e.g.\ for various
    values of~$q$ in the CM method.
    Additionally, one would expect the smaller size to make the generalized class polynomials less expensive to compute.
    Again for $C$ a modular elliptic curve with a given Weierstrass model, we present a preliminary analysis of the cost of computing a generalized class polynomial $H_{\tau}[C]$ when compared to the ``classical'' class polynomial $H_{\tau}[x]$ (though recall that the latter already dominates all previously-known class invariants along a positive density subset of discriminants for
    $C=\xplus{119}$, cf.\ Section~\ref{ssec:comparison}).
    
    \subsection{Speed of complex analytic computation}
    \label{ssec:complexanalytic}
    
    We now explain how to adapt the complex analytic approximation
    algorithm to generalized class polynomials.
    
    To compute the classical class polynomial $\Hx$ one first evaluates $x(\tau)$
    and all its conjugates, which are of the form $x_i(\tau_i)$,
    where $x_i$ and $\tau_i$ can be obtained
    using Shimura's reciprocity law~\cite{gee-stevenhagen1998}
    or $N$-systems~\cite{schertz2002weber}.
    Then one multiplies the linear polynomials $X-x_i(\tau_i)$ together
    in a binary tree using fast multiplication algorithms.
    
    As $\HC$ has roughly half the height, we only need half the precision
    at each step. This gives a great speed-up when evaluating $x_i(\tau_i)$,
    but then we also need to compute $y_i(\tau_i)$.
    Fortunately that should only take a fraction of the time required
    for computing $x_i(\tau_i)$, as we can first compute it to low precision
    and then obtain as many digits as desired quickly using
    $$ y =  \frac{-g(x) + \sqrt{g(x)^2+4f(x)}}{2}$$
    for $C : y^2 + g(x) y = f(x)$.
    
    The \textbf{binary tree} step is harder to analyze.
    Instead of having polynomials
    $A(X) = \prod_{i\in S} (X-x_i(\tau))$ to multiply for various subsets $S\subset\{1,2,\ldots, n\}$,
    we will have pairs $(F, Q)$ with
    $F = A(X)+B(X)Y$ and
    $$\div(F) = \sum_{i\in S} (P_i) + (Q) - (\#S+1)\D.$$ 
    Instead of a single multiplication $A_1A_2$
    to go from $S_1$ and $S_2$ to $S_3 = S_1\sqcup S_2$,
    we now need to compute the point $Q_3 = Q_1+Q_2$
    (with the elliptic curve group law)
    and a function $F_3$ with
    $$ \div(F_3) = \sum_{i\in S} (P_i) + (Q_3) - (\#S_3+1)\D
    = \div(F_1)+\div(F_2) + (Q_3) + (O) - (Q_1) - (Q_2).$$
    The following formula can be used:
    \begin{align}
    	F_3 &= \frac{F_1\ F_2\  R
    		\quad \mathrm{mod}\quad (Y^2-f(X))}
    	{(X-x(Q_1))(X-x(Q_2))},\quad \mbox{where}\label{eq:F3}
    	\\
    	R &=
    	(x(Q_1)-x(Q_2))\ Y\  +\  (y(-Q_2)-y(-Q_1))\ X\\
    	&\ \  + x(Q_2)y(-Q_1) - x(Q_1)y(-Q_2),
    \end{align}
    and  
    where the reduction modulo
    $Y^2-f(X)$ keeps the outcome of degree $\leq 1$ in $Y$.
    
    We can multiply $F_1$ with $F_2$ using three multiplications
    of half the degree, by the same trick that is used
    in Karatsuba multiplication.
    Indeed, let 
 $$ C = A_1A_2,\quad D = B_1B_2,\quad \mbox{and}\quad
 E=(A_1+B_2)(A_2+B_1)$$
 to get $F_1F_2 = (C+Df) + (E-C-D)Y$.
     So computing $F_3$
    involves three polynomial multiplications
    of half the degree of $F_1$ and~$F_2$,
    as well as various multiplications and long divisions
by fixed-degree polynomials and various additions
and subtractions.
The most serious computations in the binary tree are now
done with half the degree \emph{and} half the number of digits,
but three times as often, which takes $3/16$th of the time
with naive multiplication and still less than $3/4$ of the time with
quasi-linear-time multiplication. The impact of the extra additions and subtractions, as well as the extra multiplications by a linear polynomial in $X$ and $Y$ and long division by the denominator of \eqref{eq:F3} requires further analysis, but we expect this to be minor. Regardless, for large discriminants, the main bottleneck is in memory complexity (as noted in \cite[Section~7]{enge2009complexity}), and here we obtain an improvement of a factor of $1/2$ when passing from $\Hx$ to~$\HC$.
    
    \subsection{Adapting the CRT method}

    \subsubsection{Overview of CRT class polynomial computation}

    We now heuristically estimate the expected speed-up
    when computing $\F$ instead of $\Fx$ using the
    (currently state-of-the-art) CRT method for class
    polynomial computation \cite{enge-sutherland2010,sutherland2011crt,sutherland2012crt}.
    We restrict to the case of $C$ such that all $q$-expansion coefficients
    of $x$ and $y$ are rational, and will analyse some steps
    only in the main case where $C$ is a quotient of $\xplusN$.
    To keep our exposition simple, we will not treat the main improvement of~\cite{sutherland2012crt},
    even though we do expect it to combine well with our generalized class invariants.
    We plan to give a more detailed account and an implementation
in future work.
    
    For the CM method, it is more efficient to directly compute class polynomials
    modulo $q$ using
    the \emph{online CRT}
    as in \cite[Section~2]{sutherland2011crt}.
    In other words, we never write down $\F\in\ZZ[X,Y]$,
    but instead compute $(\F\ \mathrm{mod}\ q)\in\FF_q[X,Y]$
    directly
    from $(\F\ \mathrm{mod}\ p)$ for $p$ in a set $S$ of small primes.
    The space complexity of the CM method is then $n\log(q)$, which is independent
    of our choice of class function.
    The set $S$ must be chosen in such a way that $\prod_{p\in S} p$
    is larger than 4 times the largest coefficient.
    
    By cutting the number of digits
    in half when switching from $x$ to~$C$, we essentially cut $\# S$ in half.
    If the amount of work that we do for each prime $p$ does not grow too much,
    then our class function $\F$ yields a speed-up over the classical class polynomial $\Fx$.
    
    What needs to be done for each $p$ is the following.
    \begin{enumerate}
    	\item\label{it:enumerate} Enumerate all $E''$ with endomorphism ring $\mathcal{O}$
    	and compute the appropriate points in $C(\FF_p)$.
    	\item\label{it:combine} Compute $(F\ \mathrm{mod}\ p)$ by putting together the information
    	from Step \ref{it:enumerate}.
    \end{enumerate}
    In practice, for ``typical'' discriminants $D$ with 9 to 14 digits,
    Sutherland~\cite[Sections 8.3 and~8.4]{sutherland2011crt}
    finds that performing Steps \ref{it:enumerate} and~\ref{it:combine} together $\#S$ times 
    is the dominant part of the CRT method.
    
    We will now argue why we expect each of these steps to take
    (much) less than twice as long 
    with the generalized class polynomial for suitable~$C$.
    Together with the fact that our set $S$ is only half the original size
    due to the reduction factor,
    this means that computing $\F$ takes less
    time than computing~$\Fx$.
    
    \subsubsection{Enumerating via the Fricke involution}
    
    Step~\ref{it:enumerate} is already very subtle in the case of a single class invariant~$f$.
    Indeed, there could be
    multiple Galois orbits of values $f(\tau)$ for the same order~$\mc{O}$,
    and hence multiple irreducible class polynomials $H_{\tau_i}[f]\in K[X]$.
    In the CRT method, one has to make sure to compute the polynomials
    $(H_{\tau_i}[f]\ \mathrm{mod}\ p)_p$ for the same value of~$i$,
    and only for $\tau_i$ for which this is a class invariant.
    This issue is addressed in detail in \cite[Section~4]{enge-sutherland2010}.
    
    We will first explain how to adapt one solution to our
    main case of quotients $C$ of $\xplusN$ where $N$ is coprime
    to the conductor of $\mc{O}$ and $D=\mathrm{disc}(\mc{O})$
    is a square modulo~$4N$.
    
    We adapt the method of Section~4.3 of \cite{enge-sutherland2010} as follows.
    We have $\QQ(\xnulN) = \QQ(j,j_N)$, where
    $j_N(z) = j(z/N) = j(W_N z)$
    for the Fricke-Atkin-Lehner involution $W_N : z\mapsto -N/z$
    (this follows for example from \cite[Proposition~6.9]{shimura1971}).
    In particular, every function $f\in\QQ(C)$ for a quotient $C$ of $\xnulN$
    can be expressed as a rational function in $j$ and~$j_N$.
    In practice, these expressions can be quite large, but (analogously to \cite[Lemma~2]{enge-sutherland2010})
    we can also obtain the value $f(z)$ as a root of $\gcd(\Psi_{f}(X,j(z)),\Psi_{f\circ W_N}(X,j_N(z)))$
    instead.
    
    In the particular case where $C$ is a quotient of $\xplusN$, we even have
    $\QQ(C)\subset \QQ(\xplusN) = \QQ(j+ j_N,j\cdot j_N)$,
    and we can use $\Psi_{f}$ instead of $\Psi_{f\circ W_N}$.
    
    So instead of enumerating just the $j$-values, we wish to link them with the corresponding
    $j_N$-values, and we do that as follows.
    
    Suppose that $N$ is coprime to the conductor of $\mc{O}$ and that
    $D$ is a square modulo~$4N$. Then by Lemma~\ref{lem:exceptions}
    we get $a,b,c\in\ZZ$ with $a,c>0$, $b^2-4ac = D$, $N\mid c$, and
    $\gcd(ac/N,N)=\gcd(a,b,c)=1$.
    In line with Lemma~2 of \cite{enge-sutherland2010} we could even take~$c=N$
    by replacing $a$ by~$ac/N$.
    We take $z = \frac{-b+\sqrt{D}}{2a}$, $\mathfrak{n} = a\overline{z}\ZZ+N\ZZ$,
    and $\fraka = z\ZZ+\ZZ$.
    Then we have $\mc{O} = az\ZZ+\ZZ$,
    and we find that $\mathfrak{n}$ is an invertible $\mc{O}$-ideal 
    with $\mc{O}/\mathfrak{n}\cong \ZZ/N\ZZ$.
    In fact, we find $\overline{\mathfrak{n}}\fraka = z\ZZ+N\ZZ$ and hence
    $$\sigma_{[\mathfrak{n}]}j(z) = j(\mathfrak{n}^{-1}\fraka) 
    =  j(\overline{\mathfrak{n}}\fraka) = j_N(z).$$

    Exactly as in Section~4.3 of \cite{enge-sutherland2010}, we list 
    the $j$-values of elliptic curves over $\FF_p$ with
    endomorphism ring~$\mc{O}$, and arrange them into
    unoriented $[\mathfrak{n}]$-isogeny cycles.
    If $C$ is a quotient of $\xplusN$ over~$\QQ$,
    then for each edge of this graph, we find the $f$-value from the
    two $j$-values of the end points.
    (In the case where the $[\mathfrak{n}]$-isogeny cycles are 2-cycles,
    we only get one $f$-value per 2-cycle and we get a lower-degree class polynomial~$\Ff$.)
    
    In practice, we could do this for $f=x$
    exactly as in \cite{enge-sutherland2010}, and then solve for $y$
    using $\Psi_{C}(x,y,j)=0$, which is linear in~$y$.
    The only additional work compared to what is done in \cite{enge-sutherland2010}
    is computing and solving the linear equation to get~$y$, which
    is much faster than all the other steps. 
    
    In particular, Step~\ref{it:enumerate} takes much less
    than twice as long with $C$ than with~$x$,
    while we need to do it only half as often,
    which leads to a speed-up.
    Further research into these modular polynomials
    is needed in order to determine the exact gain.
    
    To also make this work for quotients of $\xnulN$ that are not quotients of $\xplusN$,
    one would need to compute oriented $[\mathfrak{n}]$-isogeny cycles.
    
\subsubsection{Other tricks for enumerating}
The methods from \cite[Sections 4.1 and~4.2]{enge-sutherland2010} also
seem amenable.
    
    The main computational tool at the beginning of Section~4.1 is the modular polynomial $\Phi_{\ell,f}$,
    which we generalize from $f$ to $C$ as follows.
    
    Let $\Phi_{\ell,C}$ be a Gr\"obner basis
    of the ideal in $\QQ[X_1,Y_1,X_2,Y_2]$ of polynomials that vanish
    on $\{(\psi(z),\psi(\ell z)) : z\in\HH\}$,
    with respect to the lexicographic ordering with $X_1>Y_1>X_2>Y_2$.
    To get from $\psi(z)$ to all possible values of $\psi(\ell z)$,
    one substitutes $\psi(z)$ for $(X_1,Y_1)$, and then solves first for
    $X_2$ and then for~$Y_2$.
    For each $C$ and $\ell$ this works for all but a finite set of primes~$p$.
    Such multivariate modular polynomials would need to be precomputed.
    One possible starting point for computing these would be~\cite{martindale2020,milio-robert2020},
    which compute multivariate (Hilbert) modular polynomials,
    each with a different method. For yet another approach to computing
    modular polynomials, see~\cite{broker-lauter-sutherland2012}.

    We expect the reduction factor to also give a reduction of the size of
    these multivariate modular polynomials, but on the other hand, we need two of them:
    one to solve for $x$ of an isogenous curve,
    and one to evaluate in $x$ and get~$y$.
    As evaluating is faster than solving, we expect the use of these modular polynomials to take much less
    than twice as long (and we need to do it only half as often, because we have half as many primes).

The `Trace Trick' of \cite[Section~4.2]{enge-sutherland2010} enables the use of the Weber function $\mathfrak{f}$
in the CRT method. In case we would also need this trick, for some more exotic
curves~$C$, we could consider applying it with arbitrary
functions $f\in \QQ(C)$ such as $f = ax+by$ for small integers $a$ and $b$.
In loc.~cit.\ the relevant trace is computed with much fewer primes,
so it is ok to apply this with the lower reduction factor of~$f$.

We did not yet consider the general algorithm of~\cite[Section~4.4]{enge-sutherland2010}.
It is the method that works for all class invariants,
but is only practical under additional restrictions.
We do not have examples of generalized class invariants
where this trick is needed. The challenging step to generalize
is factoring a large-degree function in $\QQ(C)$ in order to obtain the small
class functions.
    
    \subsubsection{Constructing a function from its roots}
    
 In the CRT setting the multiplications and long-divisions by
small-degree polynomials of Section~\ref{ssec:complexanalytic}
only take time $O(nM(\log(p)))$ per level, which is asymptotically
dominated by the $O(M(n\log(p))$ time of the multiplications
of large-degree polynomials. Therefore, 
Step \ref{it:combine} seems to take about 1.5 times as long
per prime $p$ for $\HC$ when compared to $\Hx$.
    
    \subsubsection{The total running time}

    Concluding this preliminary analysis, we estimate the cost of computing $\HC$ to be significantly lower compared to $\Hx$, though further research, in particular into (the implementation of) modular polynomials for $C$ is required to determine the exact gain. This is beyond the scope of the current paper, which focuses on introducing the generalized class functions and their height reduction.
    We plan to give a more detailed account and an implementation
    in future work.
    
	\section{General curves and bases}\label{sec:general}
	    
    Now suppose that our modular curve
    $\curve$ is not necessarily an elliptic curve.
	Let $\D$ be an effective divisor over
	$\QQ$ on $\curve$
	and let $\mc{B}=\{b_0,b_1,\ldots\}$ be a
	$\QQ$-basis of $\mc{L}(\infty\mc{D})$ ordered by ascending degree. 

	The classical case is the case where we have one modular function $f$
	and we take $\curve = \PP^1$, $\psi = f  = (f:1)$, $\D = ((1:0))=(\infty)$, and $\mc{B}=\{1,f,f^2,\ldots\}$.
	The case of all previous sections of this paper is the case where
	$\curve$ is an elliptic curve given by a Weierstrass equation, $\D = ((0:1:0))$, and $\mc{B}=\{1,x,y,x^2,xy,x^3,x^2y,\ldots\}$.
	
\begin{example}\label{example:constructbasis}
    One systematic way to choose a $\QQ$-basis of $\calL(\infty\D)$ is as follows.
	First choose $x\in \calL(\infty \D)\setminus\QQ$ of some degree $d$.
	(For example, one can take $x=f$ with $d=1$ in the classical case, and $x=x$ with $d=2$ in the elliptic case.) Now, let $y_0 = 1$ and 
		choose $y_j$ for $j = 1,2,\ldots, d-1$ in such a way
		that $$y_j \in \calL(m_j \D)\setminus\langle y_k x^e : k < j, e\in\ZZ\rangle,$$
		where $m_j$ is minimal such that this set is non-empty.
		This way we obtain a vector $\vec{y} = (y_0,\ldots, y_{d-1})$ of $d$ functions. (For example, in the classical case we have $\vec{y}=1$, and in the elliptic case we chose $\vec{y} =(1,y)$.) Then $\mc{B}=\{x^e y_j : e\in\ZZ_{\geq 0}, j\in\{0,1,2,\ldots,d-1\}\}$
is a basis of $\mathcal{L}(\infty\D)$.
We order this basis by ascending degree $de+m_j$,
and if two elements have the same degree, then we put the
one with lowest $j$ first.

\end{example}
	
	\begin{example}\label{example:hyperelliptic}
		Consider the modular curve $\xplus{191}$ (not to be confused with~$119$),
		which is hyperelliptic with model
		$t^2 = s^6+2s^4+2s^3+5s^2-6s+1$ \cite[Table~3]{galbraiththesis1996},
		and the unique cusp is at $\mc{D}=((1:1:0))$.
		One of the possible bases of $\mc{L}(\infty\mc{D})$ obtained by the recipe above is $\mc{B}=\{1,x,y_1,y_2,x^2,x^2y_1,x^2y_2,\ldots\}$, where $x=(t+s^3+s+1)/2$,  $y_1=sx$, and $y_2=s(y_1+1)$.
		The degrees of these functions are respectively $3$, $5$, and~$7$.

The function $x$ is, up to multiplicative and additive constants, equal to the Atkin function~$A_{191}$.
The reduction factors are $r(C) = 96$, $r(s) = 48$, and $r(A_{191}) = 32$.
	\end{example}

As in Section \ref{sec:section2}, let $\tau\in\HH$
imaginary quadratic and assume that $(b_i,\tau)$
is a class invariant for every $b_i\in\mc{B}$.
Then, again unique up to scaling, we obtain a non-zero
function $F_{\tau}[C,\mc{B}]=\sum_{i=0}^k a_ib_i\in K(C)$ 
($a_i\in K$) with $k$ minimal such that $\sum_{i=0}^k a_ib_i(\tau)=0$.
\begin{definition}\label{def:genclasspoly2}
We call this $F_{\tau}[C,\mc{B}]$ the \emph{generalized class function}
for the triple $C,\mc{B},\tau$.

If $\mc{B}$ is as in Example~\ref{example:constructbasis}
then we again refer to the associated polynomial
$H_{\tau}[C,\mc{B}]\in K[X,Y_1,\ldots,Y_d]$
(of total degree $\leq 1$ in $Y_1,\ldots, Y_d$ and
such that $H_{\tau}[C,\mc{B}](x,y_1,\ldots, y_d) =
F_{\tau}[C,\mc{B}]$)
as the \emph{generalized class polynomial}. 
\end{definition}

\begin{example}\label{example:otherbasis}
It turns out that, already for the case of elliptic curves, allowing the freedom of the choice of basis of may in reality lead to potentially better practical reduction factors.
Revisiting our main example $C:=\xplus{119}$,
denote by $w:=\mathfrak{w}_{7,17}$ the function \eqref{eq:defw} and by $z:=x+y$ the sum of the Weierstrass coordinates for the model \eqref{eq:weierstrass}. Now consider the basis $\mc{B}:=\{1,x,z,w,xz,wx,wz,w^2,wxz,w^2x,\ldots\}$ of $\mc{L}(\infty\mc{D})$. The resulting generalized class polynomials corresponding to the discriminants of Table~\ref{tab:119} are listed in Table~\ref{tab:119gen}.
We get practical reduction factors in Figure \ref{fig:119h100gen} that are better than those in Figure~\ref{fig:119h100}.

A likely explanation for this improvement is that now not only the poles,
but also the zeroes are as much restricted to the cusps of $\xplus{119}$ as possible.
Indeed, the points $O=(0:1:0)$ and $P=(0,0)$ are the cusps, while
$2P$ and $3P$ are rational CM points.
Now $\div(w) = 4(P)-4(O)$,
$\div(x) = (P) + (3P) - 2(O)$,
and $\div(y) = 2(P) + (2P) -3(O)$.
In particular, the function $w$ is a modular unit.
As explained in Remark~\ref{rem:changemodels}, modular units in the classical setting
give better practical reduction factors than non-units, even though the reduction factors are asymptotically the same.
\end{example}

\begin{table}[htbp]
\setlength{\tabcolsep}{3pt}
\renewcommand{\arraystretch}{1.2}
\begin{tabular}{c|c|c}
 $D$  & $n$ & $F_{\tau}[C,\mc{B}]$ \\ \hline
 $-52$ & $2$ & $z-x+1$ \\ \hline
 $-523$ & $5$ & $xw - xz - x + 3w + z$\\ \hline
 \multirow{2}{*}{$-5347$} & \multirow{2}{*}{$13$} & $xw^3 - 10xw^2z + 42xw^2 + 48w^3 + 13xwz + 35w^2z + 62xw $\\
 & &  $ + 104w^2 + 39xz + 90wz - 11x + 41w + 39z + 1$ \\ \hline
 \multirow{5}{*}{$-15139$} & \multirow{5}{*}{$29$} & $xw^7 - 33xw^6z + 5874xw^6 + 849w^7 - 2119xw^5z - 3865w^6z $\\
 & & $+ 31183xw^5 - 4249w^6 + 2200xw^4z - 15449w^5z + 36423xw^4$\\
 & & ${-}29399w^5{+}6066xw^3z{-}46282w^4z{+}46223xw^3{-}27578w^4{+}6207xw^2z$\\
 & &   $ - 30128w^3z + 31320xw^2 - 47581w^3 + 6757xwz - 35595w^2z$\\
 & & $ + 8017xw - 17181w^2 - 742xz - 10159wz - x - 2797w + 22z$\\ \hline
\end{tabular}
\\[2mm]
\caption{Some conjecturally correct generalized class functions for the curve $C=\xplus{119}$ using the $\mc{L}(\infty\mc{D})$-basis $\mc{B}:=\{1,x,z,w,xz,wx,wz,w^2,wxz,w^2x,\ldots\}$.}
\label{tab:119gen}
\end{table}

\begin{figure}[ht]
\begin{subfigure}[ht]{1\textwidth}
\centering
\includegraphics[scale=0.8]{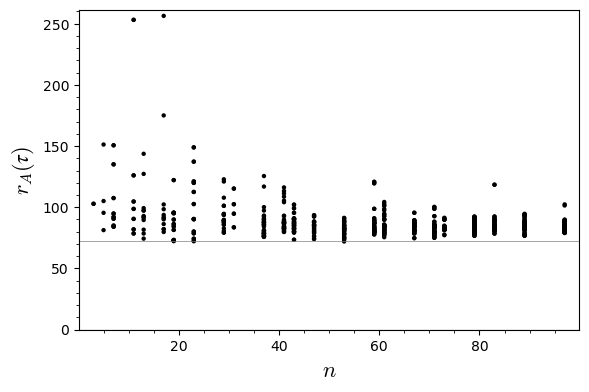}
\end{subfigure}
\vfill
\begin{subfigure}[ht]{1\textwidth}
\centering
\includegraphics[scale=0.8]{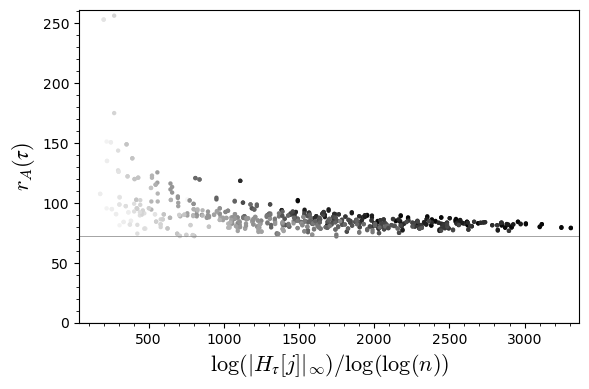}
\end{subfigure}
\caption{Practical reduction factors for $H_{\tau}[\xplus{119},\mc{B}]$ for fundamental discriminants $D$ with $\gcd(D,N)=1$ and prime class number $n<100$.}
\label{fig:119h100gen}
\end{figure}

\begin{theorem}\label{thm:reductiongeneral}
Let $\curve : y^2 + g(x) y = f(x)$ with $f,g\in\QQ[X]$ be a hyperelliptic curve such that $4f(x)+g(x)^2$ has odd degree and $\mathrm{Jac}(\curve)(\QQ)$ is finite. Set $\D$ to be the unique point at infinity and choose the basis $\mc{B}=\{1,x,x^2,y,xy,x^2y,\ldots\}$ of $\mc{L}(\infty\mc{D})$. Then Theorem \ref{thm:reductionhyper} and Proposition~\ref{prop:reductionhyper} also hold for $C$ and $H_{\tau}[C,\mc{B}]$.
	\end{theorem}
	\begin{proof}
The original proof now goes through with only the following change.
There are finitely many possibilities for the class $c$ of the divisor
$ - \sum_{\sigma} ((\sigma(\psi(\tau))) - \D)$
by our assumption that 
$\mathrm{Jac}(\curve)(\QQ)$ is finite.
For every~$c$, choose
a representative $\sum_{i=1}^{m} ((P_i)-\D)$ with $m$ minimal
and consider a primitive polynomial $T\in\ZZ[X]$ with
roots $x(P_i)$ for $i=1,\ldots, m$.
\end{proof}

\begin{remark}\label{rem:120}
Our proofs of Theorems \ref{thm:reductionhyper} and \ref{thm:reductiongeneral}
heavily rely on the fact that Heegner points are torsion.
To completely remove the assumption on ranks, one would therefore need to bound
the Heegner points, even in the rank-one case.
Moreover, the proofs rely on the hyperelliptic equation where we use
that $|a| \leq |a+bi|$ for real numbers $a$ and~$b$.
Though we expect an analogue of these results to hold for general modular curves, this would require
additional ideas. Do note that such an analogue would yield arbitrarily high reduction factors
for generalized class polynomials by \eqref{eq:reductionx0}. For example, for $C=\xplus{239}$ of genus $3$ we already obtain $r(C)=120$, exceeding the Br\"oker-Stevenhagen bound.
\end{remark}

\noindent \textbf{Conflict of interest statement.}
The authors assert that there are no conflicts of interest.
\bigbreak\noindent \textbf{Data availability statement.}
The authors declare that the data supporting the findings of this study are available within the article and its supplementary information files.

	\bibliography{bib}{}
	\bibliographystyle{plain}
\end{document}